\documentclass{article}

\usepackage{hyperref}
\usepackage[all]{xy}
\usepackage{mathrsfs}
\usepackage{amssymb}
\usepackage{amsmath}
\usepackage{latexsym}
\usepackage{amsthm}
\usepackage{graphics}
\usepackage[english]{babel}
\usepackage{graphicx}
\usepackage{etoolbox}
\usepackage{mathtools}
\usepackage{xcolor}

\usepackage[italian]{varioref}
\usepackage[utf8]{inputenc}
\usepackage[T1]{fontenc}
\usepackage{lmodern}
\usepackage{microtype}
\usepackage{pgf,tikz,tikz-cd}
\usetikzlibrary{arrows}

\usepackage[nowrite,swapnames]{frontespizio}

\newcommand{\bG}{\mathbb{G}}

\newcommand{\cE}{\mathcal{E}}
\newcommand{\cL}{\mathcal{L}}
\newcommand{\bN}{\mathbb{N}}
\newcommand{\bQ}{\mathbb{Q}}
\newcommand{\bZ}{\mathbb{Z}}

\newcommand{\bC}{\mathbb{C}}

\newcommand{\cA}{\mathcal{A}}
\newcommand{\cO}{\mathcal{O}}

\newcommand{\bP}{\mathbb{P}}

\newcommand{\isom}{\cong}

\newtheorem{defin}{Definition}[section]
\newtheorem{lemma}[defin]{Lemma}

\newtheorem{prop}[defin]{Proposition}
\newtheorem{rem}[defin]{Remark}
\newtheorem{thm}[defin]{Theorem}
\newtheorem*{thm1*}{Theorem $\ref{isoProduct}$}
\newtheorem*{thm2*}{Theorem $\ref{nonIsoProduct}$}
\newtheorem{conj}[defin]{Conjecture}
\newtheorem*{conj*}{Conjecture $\ref{conj}$}
\newtheorem{exe}[defin]{Example}

\usepackage{blindtext}

\title{Monodromy of double elliptic logarithms}
\author{Francesco Tropeano}
\date{}
\newcommand{\Addresses}{{
  \bigskip
  \footnotesize

  \noindent Francesco Tropeano\\
  Dipartimento di Matematica e Informatica,\\
  Università della Calabria, Via Pietro Bucci, Rende, Italy\\
  \noindent E-mail address: \texttt{francesco.tropeano@unical.it}\par\nopagebreak
}}

\begin{document}

\maketitle

\section{Introduction}

The following paper is devoted to the study of the monodromy of \emph{double elliptic logarithms}, i.e. a generalized notion of logarithm defined on fibered products of elliptic schemes. 

Let us consider an abelian scheme $\cA \rightarrow B$ and a section $\sigma:B \rightarrow \cA$. Period functions, abelian logarithms of $\sigma$ and the Betti map can be always globally defined on the universal cover of $B$, but they cannot be in general well-defined on the whole of $B$. We are interested in studying the minimal unramified cover on which abelian logarithm and periods become well-defined in the case of products of elliptic schemes. 

This analysis starts from a paper of Corvaja and Zannier (see \cite{CZ2}), where they study the monodromy problem in the case of a non-isotrivial elliptic scheme $\cE \rightarrow B$. In that case, we can consider the minimal unramified cover $B^* \rightarrow B$ on which periods becomes well-defined and the minimal unramified cover $B_\sigma \rightarrow B^*$ on which an elliptic logarithm of $\sigma$ becomes well-defined. They proved the following:
\begin{thm}\label{mainThm}
	Given a non-torsion (rational) section $\sigma:B\rightarrow \cE$, the cover $B_\sigma \rightarrow B^*$ has infinite degree and its Galois group is isomorphic to $\bZ^2$.
\end{thm}

In the context of abelian schemes of arbitrary relative dimension, a theorem due to Y. André \cite{A} provides, under suitable assumptions, the best possible information about the Zariski closure of the image of the monodromy representation of the fundamental group, associated to a section; here, we want to determine the relative monodromy group, which is more than the information provided passing through the Zariski closure of the monodromy group (see Remark $\ref{remAndré}$ in the present paper). This paper aims at extending Theorem $\ref{mainThm}$ to all fibered products of two non-isotrivial elliptic schemes, which are abelian schemes of relative dimension 2. At first, we consider an abelian scheme $\cA \rightarrow B$ and a section $\sigma:B \rightarrow \cA$. In analogy with the case of elliptic schemes, we can consider the minimal unramified cover $B^* \rightarrow B$ on which periods become well-defined and the minimal unramified cover $B_\sigma \rightarrow B^*$ on which an abelian logarithm of $\sigma$ becomes well-defined. We begin by stating the following conjecture (which is beyond our aims since it concerns abelian schemes of arbitrary relative dimension) and proving that it is invariant under isogeny:

\begin{conj*}
	Let $\pi: \cA \rightarrow B$ be an abelian scheme of relative dimension $g$ which has no fixed part. If the image of $\sigma: B \rightarrow \cA$ is not contained in any proper group-subscheme, then the cover $B_\sigma \rightarrow B^*$ has infinite degree and its Galois group is isomorphic to $\bZ^{2g}$.
\end{conj*}

Observe that, as in the case of elliptic schemes, the previous conjecture is stronger than André's theorem and implies in particular the algebraic independence of the logarithm of any non-torsion section and the periods.

Then, we shall consider the case where $\cA$ is a fibered product of the form $\cE_1\times_B\cE_2 \rightarrow B$, where an elliptic scheme is always assumed to be non-isotrivial.
\begin{rem}	
	If we consider a trivial elliptic scheme $E \times B \rightarrow B$, where $E$ is an elliptic curve defined over $\bC$, then the cover $B_\sigma \rightarrow B^*$ can be trivial for some non-torsion section $\sigma$. At first, let us observe that $B^*=B$, since periods can be defined on the whole of $B$. Moreover, let us consider a non-torsion point $P \in E$ and let us define the non-torsion section
	$$\sigma: b \mapsto (P,b) \qquad \textrm{for each } b \in B.$$
	In this case, the logarithm $\log_\sigma$ is a constant function. Thus, it is well-defined on the whole of $B$ and we have $B_\sigma=B^*=B$.
	
	If $\cE \rightarrow B$ is an isotrivial but non-trivial elliptic scheme then periods can be defined on the whole of $B$, so that we have again $B^*=B$. In this case the Mordell-Weil theorem for function fields predicts that the group of $\bC(B)$-rational points of the generic fiber of $\cE \rightarrow B$ is finitely generated. Thus, we can prove that a logarithm of a non-trivial section cannot be defined on $B^*=B$ by arguing in the following way: if a non-zero rational section admits an elliptic logarithm which is well defined on the whole of $B(\bC)$, then we may divide it, and hence the section, by any prescribed positive integer and again we have maps well defined on $B(\bC)$. Thus the section would be infinitely divisible on $B(\bC)$ (since the submultiples of the sections would be algebraic and well defined on $B(\bC)$, hence rational on $B(\bC)$). But this violates the Mordell-Weil theorem for the generic fiber of $\cE \rightarrow B$ over the function field of $B(\bC)$.
	
	This remark shows that the case of an isotrivial elliptic scheme is well understood. This is why we only focus on non-isotrivial elliptic schemes in what follows.
\end{rem}
Let us go back to considering fibered products of two elliptic schemes. We will distinguish two cases:
\begin{enumerate}
	\item[$\bullet$] the elliptic schemes $\cE_1 \rightarrow B$ and $\cE_2 \rightarrow B$ are isogenous;
	
	\item[$\bullet$] the elliptic schemes $\cE_1 \rightarrow B$ and $\cE_2 \rightarrow B$ are not isogenous.
\end{enumerate}
Distinguishing the two cases, a careful analysis of the relationship between periods and logarithms of sections of the two factors leads us to the two following results:

\begin{thm1*}
	Let $\sigma_1: B \rightarrow \cE_1, \sigma_2: B \rightarrow \cE_2$ be rational sections of two elliptic schemes such that at least one of them is non-torsion. Suppose that there exists an isogeny $\phi:\cE_1 \rightarrow \cE_2$. Let us consider the abelian scheme $\pi: \cA:=\cE_1\times_B\cE_2 \rightarrow B$ endowed with the (non-torsion) section $\sigma=(\sigma_1,\sigma_2)$. We have the following situation:
	\begin{enumerate}
		\item if $\phi\circ\sigma_1, \sigma_2$ are linearly dependent over $\bZ$, the cover $B_\sigma\rightarrow B^*$ has infinite degree and its Galois group is isomorphic to $\bZ^2$;
		
		\item if $\phi\circ\sigma_1, \sigma_2$ are linearly independent over $\bZ$, the cover $B_\sigma\rightarrow B^*$ has infinite degree and its Galois group is isomorphic to $\bZ^4$.
	\end{enumerate}
\end{thm1*}

\begin{thm2*}
	Let $\sigma_i: B \rightarrow \cE_i, i=1,2$ be rational sections of the given elliptic schemes and suppose they are not both torsion sections. Let us consider the abelian scheme $\pi: \cA:=\cE_1\times_B\cE_2 \rightarrow B$ endowed with the (non-torsion) section $\sigma=(\sigma_1,\sigma_2)$. We have the following situation:
	\begin{enumerate}
		\item if one between $\sigma_1, \sigma_2$ is a torsion section, the cover $B_\sigma\rightarrow B^*$ has infinite degree and its Galois group is isomorphic to $\{0\}$ or to $\bZ^2$;
		
		\item if neither $\sigma_1$ nor $\sigma_2$ is a torsion section, the cover $B_\sigma\rightarrow B^*$ has infinite degree and its Galois group is isomorphic to $\bZ^4$.
	\end{enumerate}
\end{thm2*}

These results determine the relative monodromy of abelian logarithms with respect to periods in the cases of fibered products of elliptic schemes.

\textbf{Acknowledgements.} This work was part of the author's PhD thesis. The author is grateful to Umberto Zannier and Pietro Corvaja for helpful discussions, for their kind attention, their useful advice and references. The author thanks Paolo Antonio Oliverio for useful discussions and references.

\section{Abelian schemes and logarithms of sections}

Let $\pi:\cA \rightarrow B$ be a complex abelian scheme of relative dimension $g$; here $B$ is a quasi-projective smooth curve, $\cA$ is a quasi-projective variety and $\pi:\cA \rightarrow B$ is a proper surjective morphism all of whose fibers $\cA_b:=\pi^{-1}(b)$ are abelian varieties of constant dimension $g$. We always suppose that the abelian scheme $\pi: \cA \rightarrow B$ has no fixed part and that there exists a section $\sigma_0:B \rightarrow \cA$ which marks the origin in each fiber.

Over every point $b \in B$, we have an abelian exponential map $\textrm{Lie}(\cA_b) \rightarrow \cA_b$, whose kernel is the period lattice. Since $\textrm{Lie}(\cA_b)$ is a complex vector space of dimension $g$, the period lattice can be seen as a lattice in $\bC^g$ for each $b \in B$. The family of Lie algebras $\textrm{Lie}(\cA)\rightarrow B$ defines a vector bundle over $B$ and we have the exponential map
$$\exp: \textrm{Lie}(\cA) \rightarrow \cA.$$

Any fiber $\cA_b$ is analytically isomorphic to a complex torus $\bC^g/\Lambda_b$, where $\Lambda_b$ is a lattice of (maximal) rank $2g$. On suitable open subsets $U \subset B$ in the complex topology, we can find holomorphic functions $\omega_{U,1}, \ldots, \omega_{U,2g}: U \rightarrow \bC^g$ such that $\omega_{U,1}(b), \ldots, \omega_{U,2g}(b)$ is a basis of $\Lambda_b$ for each $b \in U$. Moreover we may assume that $U$ is simply connected and that $B$ is covered by such sets.

Observe that by restricting the map $\exp$ to $U \times \bC^g$, we obtain the covering map
$$U\times \bC^g\rightarrow \cA_{|U},$$
where we are denoting $\pi^{-1}(U)$ by $\cA_{|U}$.

\begin{defin}
	Let $\sigma:B \rightarrow \cA$ be a section of the abelian scheme and let $U\subset B$ be an open set as above. A \emph{logarithm of $\sigma$ in $U$} is a lifting of $\sigma_{|U}$ to $U\times \bC^g$.
\end{defin}

In other words we have the following commutative diagram:
$$\begin{tikzcd}
	& U\times \bC^g \arrow{d}\\
	B \supset U \arrow[swap]{r}{\sigma_{|U}} \arrow[dashed]{ur}{\xi} & \pi^{-1}(U) \subset \cA.
\end{tikzcd}$$
Observe that $\xi$ is of the form $\xi=(\textrm{id},\tilde{\sigma})$; the holomorphic map $\tilde{\sigma}:U \rightarrow \bC^g$ will be called \emph{a logarithm of $\sigma$ in $U$}. By definition, saying that $\xi(b)$ is a logarithm of $\sigma(b)$ means that $\exp_b\circ \, \tilde{\sigma}(b)=\sigma(b)$ on $U$. In what follows, we will denote a logarithm of a section $\sigma$ by $\log_\sigma$ (instead of $\widetilde{\sigma})$.

\subsection{Monodromy representations}
	Given an abelian scheme $\cA \rightarrow B$, let us consider the fundamental group $G:=\pi_1(B,b)$ where $b \in B$ is a fixed base point. Given $g \in G$, we can consider the analytic continuation of the periods along any loop in $B$ belonging to the homotopy class $g$: this procedure induces a change of basis of the lattice $\Lambda_b$. In other words, the monodromy action of $G$ on periods induces a homomorphism to $\textrm{GL}_{2g}(\bZ)$; since the action preserves the orientation of the basis, the image of the said homomorphism is contained into $\textrm{SL}_{2g}(\bZ)$, so that we obtain a representation $\rho:G \rightarrow \textrm{SL}_{2g}(\bZ)$ which describes the monodromy of periods. If $\omega=u_1\omega_1 + \ldots +u_{2g}\omega_{2g}$ is a period, the monodromy action on the $\bZ$-module generated by periods is given by
	$$g\cdot \left(\begin{matrix}
		u_1\\ \vdots \\ u_{2g}
	\end{matrix}\right) = \rho(g)\left(\begin{matrix}
		u_1\\ \vdots \\ u_{2g}
	\end{matrix}\right).$$
	
	Moreover, given a section $\sigma:B \rightarrow \cA$, observe that two branches of a logarithm over $b \in B$ have to differ by an element of $\Lambda_b$: thus, fixed $g \in G$ we have that $\log_\sigma$ transforms in the following way:
	$$\log_\sigma \mapsto \log_\sigma +(u_1, \ldots, u_{2g})\cdot \left(\begin{matrix}
		\omega_1\\
		\vdots\\
		\omega_{2g}
	\end{matrix}\right),$$
	where $u_1, \ldots, u_{2g} \in \bZ$. Observe that the monodromy group of the logarithm as a function defined locally on $B^*$ is a subgroup of $\bZ^{2g}$.
	
	\begin{rem}\label{torsionSection}
		Let us consider a torsion section $\sigma:B \rightarrow \cA$. The Betti map of such a section is constant and a logarithm is a rational constant combination of periods. In other words we have
		$$\log_\sigma = q_1 \omega_1 + \cdots + q_{2g} \omega_{2g},$$
		where $q_1, \ldots, q_{2g} \in \bQ$. Therefore, a loop which leaves unchanged periods via analytic continuation, leaves also unchanged the logarithm of such a section. In other words, the cover $B_\sigma \rightarrow B^*$ is trivial in this case. This is why we only consider non-torsion sections in what follows.
	\end{rem}

In the particular case of an elliptic scheme $\pi:\cE \rightarrow B$, if we look at the simultaneous monodromy action of $G:=\pi_1(B)$ on periods and logarithm of a section $\sigma:B \rightarrow \cE$, we provide a representation
$$\theta_\sigma: G \rightarrow \textrm{SL}_3(\bZ),$$
where every matrix $\theta_\sigma(g)$ is of the form
\begin{equation}\label{simultaneousRepresentation}
	\theta_\sigma(g)=\left( \begin{matrix} T_g & w_g \\ 0 & 1 \end{matrix} \right),
\end{equation}
where $T_g=\rho(g) \in \textrm{SL}_2(\bZ)$ is a matrix acting on periods (so it does not depend on $\sigma$), and $w_g=(u_g,v_g)^t \in \bZ^2$ is a vector which corresponds to the following monodromy action on logarithm:
$$\log_\sigma \xrightarrow{g} \log_\sigma + u_g\omega_1 + v_g\omega_2.$$

Theorem $\ref{mainThm}$, which determines the relative monodromy group of the logarithm with respect to periods, can be restated in terms of representations as follows:

\begin{thm}\label{mainTheorem}
	Given a non-torsion (rational) section $\sigma: B \rightarrow \cE$, the kernel of the homomorphism $\theta_\sigma(G) \rightarrow \textrm{SL}_2(\bZ)$ is isomorphic to $\bZ^2$, which is equivalent to saying
	$$\theta_\sigma(\ker{\rho})\isom \bZ^2.$$
\end{thm}

\begin{rem}\label{remAndré}
	Observe that the conclusion of Theorem $\ref{mainTheorem}$ is stronger than knowing the kernel of the homomorphism $\theta_\sigma(G)^\textrm{Zar} \rightarrow \textrm{SL}_2$, obtained by taking the Zariski closure of the group $\theta_\sigma(G)$, as the following example shows. Define $H$ to be the subgroup of $\textrm{SL}_3(\bZ)$ generated by the matrices
	$$A:=\left( \begin{matrix}
		1 & 2 & v_1\\
		0 & 1 & v_2\\
		0 & 0 & 1
	\end{matrix}\right) =: \left( \begin{matrix}
		A_0 & v\\
		0 & 1
	\end{matrix}\right),
	\qquad
	B:=\left( \begin{matrix}
		1 & 0 & w_1\\
		2 & 1 & w_2\\
		0 & 0 & 1
	\end{matrix}\right) =: \left( \begin{matrix}
		B_0 & w\\
		0 & 1
	\end{matrix}\right),$$
	where $A_0, B_0$ are the standard unipotent generators of $\Gamma_2$ and $v,w$ are a basis for $\bZ^2$. It can be shown that the Zariski closure of $H$ is the full semidirect product of $\textrm{SL}_2$ by $\bG_a^2$, whereas the kernel of the natural map to $\textrm{SL}_2$ is trivial; the details can be found in \cite{CZ1}.
\end{rem}

\section{Monodromy of abelian logarithms: invariance under isogeny}

Let $\pi:\cA \rightarrow B$ be a complex abelian scheme of relative dimension $g$ and let us consider a section $\sigma:B \rightarrow \cA$. We will call $B^* \rightarrow B$ the minimal (unramified) cover of $B$ on which a basis for the period lattice can be globally defined; moreover, we set $B_\sigma \rightarrow B^*$ to be the minimal cover of $B^*$ on which we can define the logarithm of $\sigma$. The tower of covers is represented in the diagram
$$B_\sigma \rightarrow B^* \rightarrow B.$$
Our aim is to prove something similar to Theorem $\ref{mainTheorem}$ for double elliptic schemes (i.e. fibered products of elliptic schemes). Let us start by formulating the following conjecture (which is beyond our aims since it concerns abelian schemes of arbitrary relative dimension) and by proving it is invariant under isogeny:

\begin{conj}\label{conj}
	Let $\pi: \cA \rightarrow B$ be an abelian scheme of relative dimension $g$ which has no fixed part. If the image of $\sigma: B \rightarrow \cA$ is not contained in any proper group-subscheme, then the cover $B_\sigma \rightarrow B^*$ has infinite degree and its Galois group is isomorphic to $\bZ^{2g}$.
\end{conj}

Observe that the hypothesis for the image of $\sigma: B \rightarrow \cA$ to be not contained in any proper group-subscheme is necessary, as shown by the following example.

\begin{exe}
	Let $\pi_\cE:\cE \rightarrow B$ be an elliptic scheme with zero-section denoted by $\sigma_0$ and let $\sigma:B \rightarrow \cE$ be a non-torsion section. By Theorem $\ref{mainTheorem}$, we know that the Galois group of $B_\sigma \rightarrow B^*$ is isomorphic to $\bZ^2$. We can consider the fibered product $\pi_\cA: \cA:=\cE\times_B \cE \rightarrow B$, which gives rise to an abelian scheme of relative dimension two. If we denote by $\cE_b:=\pi_\cE^{-1}(b)$ the fiber of the elliptic scheme over a point $b$, then the fiber $\cA_b$ is given by the product $\cE_b\times \cE_b$. The morphism
	$$\widetilde{\sigma}:=(\sigma_0, \sigma): B \rightarrow \cA, \qquad b \mapsto (0_b,\sigma(b))$$
	is a section of the abelian scheme, whose image is contained in the proper group-subscheme $\sigma_0(B)\times_B\cE \rightarrow B$ of $\cA\rightarrow B$. Note that the cover $B^* \rightarrow B$ is the same for the two schemes $\cA\rightarrow B$ and $\cE \rightarrow B$. Moreover, the cover $B_{\widetilde{\sigma}} \rightarrow B^*$ is the same as the cover $B_\sigma \rightarrow B^*$, i.e. the Galois group of $B_{\widetilde{\sigma}} \rightarrow B^*$ is isomorphic to $\bZ^2$, thus in this case it is not as large as possible.
	
	More generally, similar examples can be obtained by considering a section $\widetilde{\sigma}=(\sigma_1, \sigma_2)$ where $\sigma_1, \sigma_2$ are linearly dependent sections of the elliptic scheme $\cE \rightarrow B$. The case of the abelian scheme $\cE\times_B\cE \rightarrow B$ is fully covered in Section $\ref{isoSection}$, where we prove the conjecture for product of isogenous elliptic schemes.
\end{exe}

\section*{Proof of ``Invariance under isogeny''}

Let consider two abelian schemes $\pi: \cA \rightarrow B, \pi': \cA'\rightarrow B$. Recall the following definition:

\begin{defin}
	A morphism $f: \cA \rightarrow \cA'$ of group schemes over a scheme $B$ is said to be an {\emph isogeny} if $f$ is surjective and if its kernel $\ker f$ is a flat finite group $B$-scheme.
\end{defin}

We start by showing that the above conjecture is isogeny-invariant, in other words if $\cA$ and $\cA'$ are isogenous, then the conjecture for $\cA'$ implies the conjecture for $\cA$. In order to prove this, let consider a non-torsion section $\sigma:B \rightarrow \cA$ and suppose that the theorem is true for $\cA' \rightarrow B$.

\begin{lemma}\label{isogenyKer}
	If $f:\cA \rightarrow \cA'$ is an isogeny of abelian schemes, then the map $b \mapsto \# \ker(f_{|{\cA_b}})$ is constant.
\end{lemma}

\begin{proof}
	Let us consider the $B$-scheme $\ker{f} \xrightarrow{\pi_{|\ker{f}}} B$ and recall that the fiber of $\pi_{|\ker{f}}$ over a point $b \in B$ is given by
	$$(\ker{f})_b=\ker(f_{|\cA_b})=\ker{f} \times_B \textrm{Spec } \bC(b).$$
	By definition of isogeny, the restriction
	$$\pi_{|\ker{f}} : \ker{f} \rightarrow B$$
	is a flat finite morphism. Then the map
	$$B \rightarrow \bN, \qquad b \mapsto \dim_{\mathbb{C}(b)}((\pi_*\cO_{\ker{f}})_b\otimes_{\cO_{B,b}}\mathbb{C}(b))$$
	is locally constant. Since $B$ is connected, then this function is constant, say 
	$$\dim_{\mathbb{C}(b)}((\pi_*\cO_{\ker{f}})_b\otimes_{\cO_{B,b}}\mathbb{C}(b))=q \qquad \textrm{for each } b \in B,$$
	where $q \in \bN$. Then $(\pi_*\cO_{\ker{f}})_b\otimes_{\cO_{B,b}}\mathbb{C}(b)$ is isomorphic to $\bC(b)^q$ as vector space. Since the fiber $(\ker{f})_b$ is an algebraic group (in characteristic zero), hence it is reduced, then $\ker(f_{|\cA_b})$ is a disjoint union of $q$ points.
\end{proof}

Let consider the following diagram
$$\begin{tikzcd}
	\cA \arrow{rr}{f} \arrow{rd}{\pi}& & \cA' \arrow{ld}[swap]{\pi'}\\
	&B \arrow[bend left]{lu}{\sigma} \arrow[bend right, swap, dashed]{ru}{\sigma':=f\circ \sigma}
\end{tikzcd}$$
where given $\sigma$ as above we define $\sigma':=f\circ\sigma$. Obviously, this last is a section of $\cA'\rightarrow B$ since
$$\pi'\circ \sigma' = \pi' \circ f \circ \sigma = \pi\circ \sigma =\textrm{id}_B.$$

\begin{prop}
	If $\sigma$ is a non-torsion section of $\cA\rightarrow B$, then $\sigma'$ is a non-torsion section of $\cA' \rightarrow B$.
\end{prop}

\begin{proof}
	As we have just observed, $\sigma'$ is a section of $\cA' \rightarrow B$. We prove the equivalent statement ``$\sigma'$ torsion $\Rightarrow \sigma$ torsion''. So suppose $k\sigma'=0$ for some $k$. Since $f_{|\cA_b}$ is a morphism, we have $$f(k\sigma(b))=k(f\circ\sigma(b))=k\sigma'(b)=0$$ for each $b$. This means $k\sigma(b) \in \ker(f_{|\cA_b})$ for each $b$. By Lemma $\ref{isogenyKer}$, $\ker(f_{|\cA_b})$ is a finite group of fixed order $q$ for each $b \in B$. Therefore $qk\sigma(b)=0$ for each $b$. This means $(qk)\sigma=0$; in other words $\sigma$ is torsion.\\
\end{proof}

\begin{thm}\textbf{(Invariance under isogeny)}\label{isogenyInvariance}
	Let $\cA \rightarrow B, \cA' \rightarrow B$ be two abelian schemes of relative dimension $g$ and $f: \cA \rightarrow \cA'$ an isogeny. If Conjecture $\ref{conj}$ holds for $\cA'\rightarrow B$, then it holds for $\cA \rightarrow B$.
\end{thm}

\begin{proof}
	Let $\sigma:B \rightarrow \cA$ be a non-torsion section and define, as above, $\sigma':=f\circ \sigma$. We have the following two towers of coverings:
	\begin{align*}
		B_\sigma \, \rightarrow B^*_1 \rightarrow B,\\
		B_{\sigma'} \rightarrow B^*_2 \rightarrow B,
	\end{align*}
	which correspond to the relative monodromy problems for $\log_\sigma, \log_{\sigma'}$, respectively. Since $\cA$ and $\cA'$ are isogenous, the periods of $\cA$ are related to those of $\cA'$ through a matrix in $\textrm{GL}_{2g}(\bQ)$ (this matrix does not depend on $b \in B$). In order to prove this, let us consider the fibers $\cA_b, \cA'_b$ and let us denote by $\omega_i, \omega'_i \in \bC^g$ (as row vectors), for $i=1, \ldots, 2g$, the corresponding periods. The isogeny $f$ induces an isogeny on the fibers, i.e. $f_b: \cA_b \rightarrow \cA'_b$. So there exists $M (=M_b) \in \textrm{GL}_{g}(\bC)$ such that
	\begin{align*}
		\omega_1 \cdot M &= a_{1,1}\omega'_1 + \ldots + a_{1,2g}\omega'_{2g},\\
		\vdots\\
		\omega_{2g} \cdot M &= a_{2g,1}\omega'_1 + \ldots + a_{2g,2g}\omega'_{2g},
	\end{align*}
	where $a_{i,j} \in \bZ$ for each $i,j$. Thus we obtain the relation:
	\begin{equation}\label{abelianPeriods}
		\left(\begin{matrix}\omega_1M\\ \vdots\\  \omega_{2g}M\end{matrix}\right) = \zeta \left(\begin{matrix}\omega'_1\\ \vdots \\ \omega'_{2g}\end{matrix}\right),
	\end{equation}
	where we denote by $\zeta$ the matrix $(a_{i,j})_{i,j=1, \ldots, 2g} \in \textrm{GL}_{2g}(\bQ)$.
	
	\begin{rem}
		Let us consider an isogeny $f:A_1 \rightarrow A_2$ between complex abelian varieties of dimension $g$. Then we have a commutative diagram
		$$\begin{tikzcd}
			\bC^g \arrow{r}{\varphi} \arrow{d} & \bC^g \arrow{d}\\
			\bC^g/\Lambda_1 \arrow{r}{f} & \bC^g/\Lambda_2,
		\end{tikzcd}$$
		where $\varphi$ is an isomorphism obtained by covering theory. The isomorphism $\varphi$ can be expressed by right multiplication by a matrix $M \in \textrm{GL}_{g}(\bC)$. Let us prove that $M$ is uniquely determined by $f$. In fact, if two matrices $M,N$ induce the same isogeny, we have that
		$$z\cdot M \equiv z\cdot N \qquad (\textrm{mod }\Lambda_2), \qquad \qquad \textrm{for all } z \in \bC^g.$$
		Hence the map $z \mapsto z\cdot (M-N)$ sends $\bC^g$ to $\Lambda_2$. Since $\Lambda_2$ is discrete, the map must be constant. This imply $M=N$.
	\end{rem}

	\begin{rem}	
		
		By the previous Remark, the matrix $M=M_b$ considered above is uniquely determined by $f_b$; moreover, the function $b \mapsto M_b$ is a holomorphic function on $B$. In fact, we can consider the following diagram
		$$\begin{tikzcd}
			U\times \bC^g \arrow{r}{\varphi} \arrow{d}{\exp} \arrow[dashed]{rd}{\psi}& U\times \bC^g \arrow{d}{\exp'}\\
			\cA_{|U} \arrow{r}{f} & \cA'_{|U},
		\end{tikzcd}$$
		where $U$ is a simply connected open set, $\psi:=f\circ\exp$ and $\varphi$ is a lift of $\psi$ (it exists because $U\times \bC^g$ is simply connected). We necessarily have $$\varphi_{|\{b\}\times\bC^g}=[M_b],$$ where $[M_b]$ is the right multiplication by $M_b$. Since $\varphi$ is holomorphic, so is $b\mapsto M_b$. Moreover, since $M_b$ is uniquely determined by $f_b$, then the function $b\mapsto M_b$ cannot have non-trivial monodromy along loops. So it is well-defined on the whole of $B$.
	\end{rem}

	Now, let us return to the equation $(\ref{abelianPeriods})$. In particular, it means that the monodromy action of $\pi_1(B)$ on periods of $\cA$ is determined by the monodromy action on the periods of $\cA'$. To be more precise, let us consider a period $\omega'$ with coordinate vector $(u'_1, \ldots, u'_{2g})$ with respect to the basis $\omega'_i$, i.e.
	$$\omega':=(u'_1, \ldots, u'_{2g})\cdot\left(\begin{matrix}
		\omega'_1\\
		\vdots\\
		\omega'_{2g}
	\end{matrix}\right).$$
	By $(\ref{abelianPeriods})$, we obtain
	$$\omega'=(u'_1, \ldots, u'_{2g})\zeta^{-1}\left(\begin{matrix}
		\omega_1M\\
		\vdots\\
		\omega_{2g}M
	\end{matrix}\right).$$
	Let us look at the monodromy action on the bases $\omega'_i$ and $\omega_i$; denote by $\rho, \rho'$ the monodromy representations of $\cA, \cA'$ respectively. Since $M_b$ varies holomorphically and without monodromy with respect to $b \in B$ and $\zeta$ does not depend on $b\in B$, we obtain
	\begin{align*}
		h\cdot \omega' &= h \cdot (u'_1, \ldots, u'_{2g})\left(\begin{matrix}\omega'_1 \\ \vdots \\ \omega'_{2g}\end{matrix}\right)= (u'_1, \ldots, u'_{2g})\rho'(h)^T \left(\begin{matrix}\omega'_1 \\ \vdots \\ \omega'_{2g}\end{matrix}\right),\\
		h\cdot \omega' &= h \cdot (u'_1, \ldots, u'_{2g})\zeta^{-1}\left(\begin{matrix}\omega_1M \\ \vdots \\ \omega_{2g}M\end{matrix}\right)= (u'_1, \ldots, u'_{2g})\zeta^{-1}\rho(h)^T \left(\begin{matrix}\omega_1M \\ \vdots \\ \omega_{2g}M\end{matrix}\right).
	\end{align*}
	Combining the previous relations, we obtain
	\begin{align*}
		(u'_1, \ldots, u'_{2g})\rho'(h)^T\left(\begin{matrix}\omega'_1 \\ \vdots \\ \omega'_{2g}\end{matrix}\right) &= (u'_1, \ldots, u'_{2g})\zeta^{-1}\rho(h)^T \left(\begin{matrix}\omega_1M \\ \vdots \\ \omega_{2g}M\end{matrix}\right)=\\
		&= (u'_1, \ldots, u'_{2g}) \zeta^{-1}\rho(h)^T\zeta \left(\begin{matrix}\omega'_1 \\ \vdots \\ \omega'_{2g}\end{matrix}\right),
	\end{align*}
	for all $u'_1, \ldots, u'_{2g} \in \bZ$. In other terms, the two representations are conjugated between them, i.e. $\rho'(h)=\zeta^T\rho(h)(\zeta^T)^{-1}$. Therefore the periods of $\cA$ are defined over a cover $B^* \rightarrow B$ if and only if the periods of $\cA'$ are. In other words, we have $B_1^*=B_2^*$.
	
	Now, let's study the logarithms of the two abelian schemes. Let $U\subset B$ be a simply connected open set and consider $\log_\sigma, \log_{\sigma'}$:
	
	$$\begin{tikzcd}
		& U\times \bC^g \arrow{d}{\exp} \arrow{r}{\varphi} & U\times \bC^g \arrow{d}{\exp'} & \\
		B \supset U \arrow{r}[swap]{\sigma} \arrow[dashed]{ur} {(\textrm{id},\log_\sigma)} & \cA_{|U} \arrow{r}{f} & \cA'_{|U} & U \subset B \arrow{l}{\sigma'} \arrow[dashed,swap]{ul}{(\textrm{id},\log_{\sigma'})}.
	\end{tikzcd}$$
	
	As said above, the periods of $\cA$ are related to those of $\cA'$ through a matrix $\zeta$ in $\textrm{GL}_{2g}(\bQ)$; we continue to use the above notations. The induced isogeny $f_b$ is given by right multiplication with a matrix $M_b \in \textrm{GL}_{g}(\bC)$; in other words we have the commutative diagram
	$$\begin{tikzcd}
		\bC^{g} \arrow{r}{\cdot M_b} \arrow{d} & \bC^{g} \arrow{d}\\
		\cA_b \arrow{r}{f_{|\cA_b}} & \cA'_b.
	\end{tikzcd}$$
	This means that we can choose $\log_{\sigma'}(b)=\log_\sigma(b)\cdot M_b$, for a point $b \in B$. Fixed $h \in G$ let us recall that $\log_\sigma$ transforms in the following way:
	$$\log_\sigma \mapsto \log_\sigma +(u_1, \ldots, u_{2g})\cdot \left(\begin{matrix}\omega_1 \\ \vdots \\ \omega_{2g}\end{matrix}\right).$$ Moreover, as remarked above, $M_b=M$ varies holomorphically with respect to $b$ and there is no monodromy action on it.
	Therefore, for $\sigma'$ we have
	\begin{align*}
		\log_{\sigma'} = \log_\sigma\cdot M_b &\mapsto \left(\log_\sigma +(u_1, \ldots, u_{2g})\cdot \left(\begin{matrix}\omega_1 \\ \vdots \\ \omega_{2g}\end{matrix}\right)\right)\cdot M_b=\\
		=& \log_\sigma\cdot M +(u_1, \ldots, u_{2g})\cdot \left(\begin{matrix}\omega_1M \\ \vdots \\ \omega_{2g}M\end{matrix}\right)=\\
		=& \log_{\sigma'} + (u_1, \ldots, u_{2g})\zeta \left(\begin{matrix}\omega'_1 \\ \vdots \\ \omega'_{2g}\end{matrix}\right).
	\end{align*}
	Thus, we obtain that the monodromies of logarithms are related in the following way:
	$$\left(\begin{matrix} u'_1\\ \vdots \\u'_{2g}\end{matrix}\right) = \zeta^T \cdot \left(\begin{matrix} u_1\\ \vdots \\u_{2g}\end{matrix}\right),$$
	where $u_i$ and $u'_i$ describe the monodromy of $\log_\sigma, \log_{\sigma'}$, respectively.
	It follows that $B_\sigma=B_{\sigma'}$; this means that if Conjecture $\ref{conj}$ holds for $\cA' \rightarrow B$, then it holds for $\cA \rightarrow B$.\\
\end{proof}

\section{Monodromy of double elliptic logarithms}

Now, we shall analyse \emph{double elliptic schemes} over a same base, usually a curve. Such a scheme may be seen as a fiber product of elliptic schemes, or a pair of elliptic curves defined over a function field of a (same) curve.

To settle things in precise terms, let us suppose to be given two non-isotrivial elliptic schemes $\phi_i: \cE_i \rightarrow B$, for $i=1,2$, over a same base $B$ supposed to be an affine (ramified) cover of $S:= \bP_1 - \{0,1,\infty\}$. By taking a cover of $B$ if necessary, we may assume that these elliptic schemes are pullbacks of the Legendre scheme (i.e. the elliptic scheme $\cL \rightarrow S$ defined by the equation $y^2z=x(x-z)(x-\lambda z), \lambda \in S$). Each of these elliptic schemes has associated periods and monodromy action of $\pi_1(B)$ on the corresponding periods: this action yields subgroups $G_1, G_2$ of $\Gamma_2 \subset \textrm{SL}_2(\bZ)$, both of finite index. In other words we have the corresponding monodromy representations:
\begin{align*}
	\rho_{\cE_1}: &\pi_1(B) \rightarrow G_1 \subset \Gamma_2 \subset \textrm{SL}_2(\bZ),\\
	\rho_{\cE_2}: &\pi_1(B) \rightarrow G_2 \subset \Gamma_2 \subset \textrm{SL}_2(\bZ).
\end{align*}
This setting is equivalent to considering the abelian scheme $\cA:=\cE_1 \times_B \cE_2 \rightarrow B$. Hence, putting together what we said about $\cE_1, \cE_2$, we have a representation
\begin{equation}\label{representationDoubleSchemes}
	\rho=(\rho_{\cE_1},\rho_{\cE_2}): \pi_1(B) \rightarrow G_1\times G_2 \subset \Gamma_2 \times \Gamma_2 \subset \textrm{SL}_2(\bZ)\times \textrm{SL}_2(\bZ),
\end{equation}
where we identify $\textrm{SL}_2(\bZ)\times \textrm{SL}_2(\bZ)$ as a subgroup of $\textrm{SL}_4(\bZ)$. Observe that $\rho$ is exactly the monodromy representation associated with $\cA$, so we denote it by $\rho_\cA$. Thus, we have
$$\rho_\cA(g)=\left(\begin{matrix}
	\rho_{\cE_1}(g) & 0\\
	0 & \rho_{\cE_2}(g)
\end{matrix}\right).$$

If $\cE_1\rightarrow B, \cE_2 \rightarrow B$ are isogenous elliptic schemes (we may assume the isogeny to be defined over $\bC(B)$), then the periods of $\cE_1$ are related to those of $\cE_2$ through a matrix in $\textrm{GL}_2(\bQ)$. This reflects in the fact that there exists a constant matrix $\zeta \in \textrm{GL}_2(\bQ)$ such that $\rho_{\cE_2}=\zeta^{-1}\rho_{\cE_1}\zeta$. Thus, in particular, the image of $\rho_\cA$ is a graph, and the same holds for its Zariski closure in $\textrm{SL}_2(\bC) \times \textrm{SL}_2(\bC)$, so we may express this by considering it to be `small'. The following theorem, which we only state (for a proof see \cite{CZ1}), establishes a converse assertion, namely whether a `small' image necessarily implies the existence of an isogeny.

\begin{thm}\textbf{(Isogeny theorem)}\label{isogenyThm}
	Let $\cE_1, \cE_2$ be elliptic schemes over $B$, as above, and consider the monodromy representations as above. Then either the Zariski closure of the image of $\rho$ is the whole of $\textrm{SL}_2\times \textrm{SL}_2$, or $\cE_1, \cE_2$ are isogenous (over a cover of $B$) and there exists $\sigma \in \textrm{GL}_2(\bQ)$ such that $\rho_2(g)=\sigma^{-1}\rho_1(g)\sigma$ for all $g \in \pi_1(B)$.
	
	Also, for large enough prime number $p$, either the image of $\rho$ in $\textrm{SL}_2(\bZ_p)\times \textrm{SL}_2(\bZ_p)$ is dense in the whole group or we fall into the same conclusion.
\end{thm}

Now, let us consider two sections $\sigma_1$ and $\sigma_2$ of $\cE_1 \rightarrow B$ and $\cE_2 \rightarrow B$, respectively. This setting is equivalent to consider the abelian scheme $\cA:=\cE_1 \times_B \cE_2 \rightarrow B$ with a section $\sigma$, whose components are $\sigma_1, \sigma_2$. Any loop $\gamma$ whose homotopy class $g$ is in $\pi_1(B,b_0)$, gives rise to a matrix $\rho_\cA(g) \in \textrm{SL}_2(\bZ)\times \textrm{SL}_2(\bZ)$ which describes the monodromy of periods and to a column vector $w_g \in \bZ^4$ which describes the monodromy of logarithm. Thus we have a representation of the fundamental group $\pi_1(B)$ in $\textrm{SL}_5(\bZ)$, given by
\begin{align*}
	\theta_\sigma: \pi_1(B) &\rightarrow \textrm{SL}_5(\bZ)\\
	g &\mapsto \left(\begin{matrix}
		\rho_\cA(g) & w_g\\
		0 & 1
	\end{matrix}\right),
\end{align*}
where
$$\rho_\cA(g)=\left(\begin{matrix}
	\rho_{\cE_1}(g) & 0\\
	0 & \rho_{\cE_2}(g)
\end{matrix}\right), \quad w_g=\left(\begin{matrix}
	u_{1,g}\\
	u_{2,g}\\
	v_{1,g}\\
	v_{2,g}
\end{matrix}\right).$$
Observe that the logarithms $\log_{\sigma_1}, \log_{\sigma_2}$ transform in the following way
\begin{align*}
	\log_{\sigma_1} &\xmapsto{g} \log_{\sigma_1} + u_{1,g}\omega_{1,\cE_1} + u_{2,g}\omega_{2,\cE_1},\\
	\log_{\sigma_2} &\xmapsto{g} \log_{\sigma_2} + v_{1,g}\omega_{1,\cE_2} + v_{2,g}\omega_{2,\cE_2},
\end{align*}
where $\omega_{1,\cE_i}, \omega_{2,\cE_i}$ denote the periods of $\cE_i \rightarrow B$. Moreover, we have
$$w_{gh}=w_g + \rho_\cA(g)\cdot w_h= \left(\begin{matrix} 
	\left(\begin{matrix} u_{1,g}\\ u_{2,g}\end{matrix}\right) + \rho_{\cE_1}(g)\cdot \left(\begin{matrix} u_{1,h}\\ u_{2,h}\end{matrix}\right)\\
	\\
	\left(\begin{matrix} v_{1,g}\\ v_{2,g}\end{matrix}\right) + \rho_{\cE_2}(g)\cdot \left(\begin{matrix} v_{1,h}\\ v_{2,h}\end{matrix}\right)
\end{matrix}\right).$$

Finally, let us recall the following notations for the elliptic schemes $\cE_i \rightarrow B, i=1,2$:
\begin{align*}
	\theta_{\sigma_i}: \pi_1(B) &\rightarrow \textrm{SL}_3(\bZ)\\
	g &\mapsto \left(\begin{matrix}
		\rho_{\cE_i}(g) & w_{i,g}\\
		0 & 1
	\end{matrix}\right),
\end{align*}
where
$$w_{1,g}:=\left(\begin{matrix}u_{1,g}\\u_{2,g}\end{matrix}\right), \quad w_{2,g}:=\left(\begin{matrix}v_{1,g}\\v_{2,g}\end{matrix}\right).$$

\subsection{Case 1: product of isogenous elliptic schemes}\label{isoSection}
	
	In this section, we shall formulate a result on the monodromy of the logarithm of a section $\sigma$ in the case in which $\cE_1, \cE_2$ are isogenous; we use the above notation $\cA=\cE_1 \times_B \cE_2$. In this case the monodromy representations are conjugate (see Theorem $\ref{isogenyThm}$), so we have $$\ker{\rho_\cA}=\ker{\rho_{\cE_1}}=\ker{\rho_{\cE_2}}.$$
	
	Moreover, in what follows we will make the following identifications: if $g \in \ker{\rho_{\cA}}$ we identify $\theta_{\sigma_i}(g) \equiv w_{i,g} \in \bZ^2$ and $\theta_\sigma(g) \equiv w_g \in \bZ^4$. Moreover, we define
	$$H_1 := \theta_{\sigma_2}(\ker{\theta_{\sigma_1}}), \quad H_2 := \theta_{\sigma_1}(\ker{\theta_{\sigma_2}}).$$
	Now, we are ready for the results of this section.
	 
	\begin{lemma}
		The groups $H_1, H_2$ are isomorphic to either $\{0\}$ or $\bZ^2$.
	\end{lemma}
	
	\begin{proof}
		At first, observe that $H_1, H_2$ are subgroups of $\bZ^2$; so they are either isomorphic to $\{0\}$ or $\bZ$ or $\bZ^2$. We want to prove that the case $\bZ$ is excluded. We give the proof for $H_1$ since the other case is analogous. Suppose by contradiction that $H_1$ is infinite cyclic: this means that for every $h \in \ker{\theta_{\sigma_1}}$, the logarithm $\log_{\sigma_2}$ of $\sigma_2$ is transformed by $h$ as
		$$\log_{\sigma_2} \xmapsto{h} \log_{\sigma_2} + \chi(h)\omega_{\sigma_2},$$
		for a fixed non-zero period $\omega_{\sigma_2}$ and a homomorphism $\chi:H \rightarrow \bZ$. In particular, let us choose $h$ such that $\chi(h)=1$. Recall that, for $g \in G=\pi_1(B)$, the logarithm $\log_{\sigma_2}$ will be sent by $g$ to a new determination of the form $$\log_{\sigma_2} + v_{1,g}\omega_1 + v_{2,g} \omega_2,$$ where $v_{1,g}, v_{2,g}$ are integers. Recall that the group $G_2=\rho_{\cE_2}(G)$ acts irreducibly on the lattice of periods, since it is Zariski-dense in $\textrm{SL}_2(\bZ)$. Then there exists $g \in G$ such that $\omega_{\sigma_2}$ is not an eigenvector of $\rho_{\cE_2}(g)$. Observe that $\ker{\theta_{\sigma_1}} \unlhd G$ and calculate the action of the element $h'=g^{-1}hg \in \ker{\theta_{\sigma_1}}$, where $g, h$ are the ones just considered. We have
		\begin{align*}
			\log_{\sigma_2} &\xmapsto{g} \log_{\sigma_2} + v_{1,g} \omega_1 + v_{2,g} \omega_2 \xmapsto{h} \log_{\sigma_2} + v_{1,g}\omega_1 + v_{2,g} \omega_2 + \omega_{\sigma_2}\\
			& \xmapsto{g^{-1}} \log_{\sigma_2} +\rho_{\cE_2}(g^{-1})\omega_{\sigma_2}.
		\end{align*}
		In other words this means $\rho_{\cE_2}(g^{-1})\omega_\sigma=\chi(h')\omega_\sigma$, but this is a contradiction since $\omega_\sigma$ is not an eigenvector of $\rho_{\cE_2}(g)$ (nor of $\rho_{\cE_2}(g^{-1})$). This concludes the proof for $H_1$.\\
	\end{proof}
	
	\begin{prop}\label{possibleCasesIso}
		Suppose that at least one between $\sigma_1, \sigma_2$ is non-torsion. The group $\theta_\sigma(\ker{\rho_\cA})$ is isomorphic to either $\bZ^2$ or $\bZ^4$.
	\end{prop}

	\begin{proof}
		We prove the theorem supposing that $\sigma_1$ is non-torsion.
		
		Recall that by Theorem $\ref{mainTheorem}$, we have $\theta_{\sigma_1}(\ker{\rho_{\cE_1}}) \isom \bZ^2$. Since $\ker{\rho_\cA}=\ker{\rho_{\cE_1}}$, then we have $2 \le \textrm{rank}\, \theta_\sigma(\ker{\rho_\cA}) \le 4$. By the previous lemma, we only have two possibilities for $H_1$, i.e. $H_1 \isom \{0\}$ or $H_1\isom \bZ^2$.\\
		
		\textbf{Case 1: $H_1\isom\{0\}$}\\
		The condition $H_1\isom\{0\}$ means that for each $g \in \ker{\rho_\cA}$ if $u_{1,g}=u_{2,g}=0$, then $v_{1,g}=v_{2,g}=0$, where the notation is the same as above. Let us prove that $\textrm{rank}\, \theta_\sigma(\ker{\rho_\cA})=2$ by proving that any three elements of the form $w_g, w_h, w_k \in \theta_\sigma(\ker{\rho_\cA})$ are linearly dependent on $\bZ$.
		
		Since $\theta_{\sigma_1}(\ker{\rho_{\cE_1}}) \isom \bZ^2$, given any three elements $g,h,k \in \ker{\rho_\cA} \subset \ker{\rho_{\cE_1}}$, there always exists $n_g, n_h, n_k \in \bZ$, not all zero, such that
		$$n_g\left(\begin{matrix}
			u_{1,g}\\
			u_{2,g}
		\end{matrix}\right) + n_h\left(\begin{matrix}
			u_{1,h}\\
			u_{2,h}
		\end{matrix}\right) + n_k\left(\begin{matrix}
			u_{1,k}\\
			u_{2,k}
		\end{matrix}\right)= \left(\begin{matrix}
			0\\
			0
		\end{matrix}\right).$$
		Thus we have
		$$\theta_\sigma(k^{n_k}h^{n_h}g^{n_g})=\left(\begin{matrix}
			\begin{matrix}
				\textrm{id}_4
			\end{matrix} & \begin{matrix}
			0\\
			0\\
			n_gv_{1,g} + n_hv_{1,h} +n_kv_{1,k}\\
			n_gv_{2,g} + n_hv_{2,h} +n_kv_{2,k}
		\end{matrix}\\
		0 & 1
		\end{matrix}\right).$$
		For what we observed at the beginning of this proof, $H_1\isom\{0\}$ implies that
		$$n_g\left(\begin{matrix}
			v_{1,g}\\
			v_{2,g}
		\end{matrix}\right) + n_h\left(\begin{matrix}
			v_{1,h}\\
			v_{2,h}
		\end{matrix}\right) + n_k\left(\begin{matrix}
			v_{1,k}\\
			v_{2,k}
		\end{matrix}\right) = \left(\begin{matrix}
			0\\
			0
		\end{matrix}\right).$$
		Therefore any three elements of $\bZ^4$ of the form $w_g, w_h, w_k$ are linearly dependent on $\bZ$, so $\theta_\sigma(\ker{\rho_\cA})\isom \bZ^2$.\\
		
		\textbf{Case 2: $H_1\isom \bZ^2$}\\
		Observe that this condition says that $\sigma_2$ is a non-torsion section, too.
		
		Since $H_1\isom \bZ^2$, we can consider a $\bZ$-basis for it and the following corresponding elements of $\theta_\sigma(\ker{\rho_\cA})$:
		$$w_3=\left(\begin{matrix}
			0\\
			0\\
			v_{1,k}\\
			v_{2,k}
		\end{matrix}\right), w_4=\left(\begin{matrix}
			0\\
			0\\
			v_{1,l}\\
			v_{2,l}
		\end{matrix}\right),
		\qquad \textrm{where } k,l \in \ker{\theta_{\sigma_1}}\subset\ker{\rho_\cA}.$$
		Since $\theta_{\sigma_1}(\ker{\rho_{\cE_1}})$ has also rank two, let us choose a $\bZ$-basis $\left(\begin{matrix}
			u_{1,g}\\
			u_{2,g}
		\end{matrix}\right), \left(\begin{matrix}
			u_{1,h}\\
			u_{2,h}
		\end{matrix}\right)$ of it, where $g,h \in \ker{\rho_\cA}$, and consider the corresponding elements of $\theta_\sigma(\ker{\rho_\cA})$:
		$$z_1=\left(\begin{matrix}
			u_{1,g}\\
			u_{2,g}\\
			v_{1,g}\\
			v_{2,g}
		\end{matrix}\right), z_2=\left(\begin{matrix}
			u_{1,h}\\
			u_{2,h}\\
			v_{1,h}\\
			v_{2,h}
		\end{matrix}\right).$$
		With an appropriate linear combination of $z_1, z_2, w_3, w_4$ we obtain that
		$$w_1=\left(\begin{matrix}
			u_{1,g}\\
			u_{2,g}\\
			0\\
			0
		\end{matrix}\right), w_2=\left(\begin{matrix}
			u_{1,h}\\
			u_{2,h}\\
			0\\
			0
		\end{matrix}\right)$$
		are elements of $\theta_\sigma(\ker{\rho_\cA})$. Moreover, $w_1, w_2, w_3, w_4$ are linearly independent over $\bZ$. Thus $\theta_\sigma(\ker{\rho_\cA}) \isom \bZ^4$.		
\end{proof}

	\subsubsection{Main Theorem}
	
	Recall that we are considering an abelian scheme $\cA:= \cE_1 \times_B \cE_2$, where we assume that $\cE_1$ and $\cE_2$ are isogenous, i.e. there exists an isogeny $\phi:\cE_1 \rightarrow \cE_2$. This isogeny induces an isogeny $\phi_\cA:=(\phi,\textrm{id}_{\cE_2}): \cE_1\times_B\cE_2 \rightarrow \cE_2\times_B\cE_2$. Since our theorem is invariant under isogeny, we can just study the case $\cA:=\cE\times_B \cE$, where $\cE \rightarrow B$ is an elliptic scheme.	
	
	\begin{thm}\label{squareProduct}
		Let $\sigma_1, \sigma_2: B \rightarrow \cE$ be rational sections of an elliptic scheme and suppose that at least one of the sections $\sigma_1, \sigma_2$ is non-torsion. Let us consider the abelian scheme $\pi: \cA:=\cE\times_B\cE \rightarrow B$ endowed with the (non-torsion) section $\sigma=(\sigma_1,\sigma_2)$. We have the following situation:
		\begin{enumerate}
			\item if $\sigma_1, \sigma_2$ are linearly dependent over $\bZ$, the cover $B_\sigma\rightarrow B^*$ has infinite degree and its Galois group is isomorphic to $\bZ^2$;
			
			\item if $\sigma_1, \sigma_2$ are linearly independent over $\bZ$, the cover $B_\sigma\rightarrow B^*$ has infinite degree and its Galois group is isomorphic to $\bZ^4$.
		\end{enumerate}
	\end{thm}
	
	\begin{proof} 
		
		\begin{enumerate}
			\item We are supposing that $\sigma_1, \sigma_2$ are linearly dependent over $\bZ$. So there exists $n_1, n_2 \in \bZ$ such that $$n_1\sigma_1 + n_2\sigma_2=0.$$ Now let us consider the corresponding elliptic logarithms $\log_{\sigma_1}, \log_{\sigma_2}$. On some domain $U \subset B$ on which they are well-defined, the linear dependence relation between the sections induces the following relation: $$n_1\log_{\sigma_1}+n_2\log_{\sigma_2}= \omega,$$ where $\omega(b) \in \Lambda_b$ is a period for each $b \in U$. By Theorem $\ref{mainTheorem}$ we know that $\theta_{\sigma_1}(\ker{\rho_\cE})\isom \bZ^2$. So let us fix a loop $\alpha$ in $B$ whose homotopy class is $g \in \ker{\rho_\cE}$ and also denote
			$$w_{1,g}=\theta_{\sigma_1}(g)=\left(\begin{matrix}
				u_{1,g}\\
				u_{2,g}
			\end{matrix}\right),
			\quad
			w_{2,g}=\theta_{\sigma_2}(g)=\left(\begin{matrix}
				v_{1,g}\\
				v_{2,g}
			\end{matrix}\right)\in \bZ^2.$$
			Now we analytically continue the relation $n_1\log_{\sigma_1}+n_2\log_{\sigma_2}= \omega$ along $\alpha$, by considering that $\omega$ remain unchanged since $g \in \ker{\rho_\cE}$. So we obtain
			$$n_1\log_{\sigma_1} +n_1u_{1,g}\omega_1 + n_1u_{2,g} \omega_2 +n_2\log_{\sigma_2} + n_2v_{1,g}\omega_1 + n_2v_{2,g} \omega_2 = \omega.$$
			Therefore we have $$(n_1u_{1,g} + n_2v_{1,g})\omega_1 + (n_1u_{2,g} + n_2v_{2,g})\omega_2=0,$$ which is the same as writing
			$$n_1\left(\begin{matrix}
				u_{1,g}\\
				u_{2,g}
			\end{matrix}\right) +	n_2\left(\begin{matrix}
				v_{1,g}\\
				v_{2,g}
			\end{matrix}\right) = \left(\begin{matrix}
				0\\
				0
			\end{matrix}\right).$$
			By arbitrariness of $g$ we have $n_1w_{1,g}+n_2w_{2,g}=0$ for all $g\in \ker{\rho_\cA}$. In other words we have $$w_{2,g} = -\frac{n_1}{n_2}w_{1,g}.$$ Then the map
			\begin{align*}
				\left(\begin{matrix}
					w_{1,g}\\
					w_{2,g}
				\end{matrix}\right) \longmapsto w_{1,g}
			\end{align*}
			is an isomorphism, so $$\theta_\sigma(\ker{\rho_\cA})\isom \bZ^2$$ and the first part is proved.
			
			\item Now, let $\sigma_1, \sigma_2$ be linearly independent sections and let us suppose by contradiction that $\theta_\sigma(\ker{\rho_\cA})$ is not isomorphic to $\bZ^4$. Let us introduce the following notations:
			\begin{align*}
			K&:=\theta_\sigma(\ker{\rho_\cA})= \left\{\left(\begin{matrix}
				u_{1,g}\\
				u_{2,g}\\
				v_{1,g}\\
				v_{2,g}
			\end{matrix}\right): g \in \ker{\rho_\cA}\right\}\\
			K_1&:= \theta_{\sigma_1}(\ker{\rho_\cA})=\left\{\left(\begin{matrix}
				u_{1,g}\\
				u_{2,g}
			\end{matrix}\right): g \in \ker{\rho_\cA}\right\}\\
			K_2&:= \theta_{\sigma_2}(\ker{\rho_\cA})=\left\{\left(\begin{matrix}
				v_{1,g}\\
				v_{2,g}
			\end{matrix}\right): g \in \ker{\rho_\cA}\right\}.		
			\end{align*}
			
			By Proposition $\ref{possibleCasesIso}$, we have that $K \isom \bZ^2$. Moreover, by Theorem $\ref{mainTheorem}$, since $\sigma_1, \sigma_2$ are linearly independent hence in particular non-torsion, we also have $K_1 \isom \bZ^2$ and $K_2 \isom \bZ^2$.

			\textbf{Claim 1: } There exists a matrix $M \in \textrm{GL}_2(\bQ)$ such that
		$$M\cdot \left(\begin{matrix}
				u_{1,g}\\
				u_{2,g}\end{matrix}\right) = \left(\begin{matrix}
				v_{1,g}\\
				v_{2,g}\end{matrix}\right)$$
		for all $g \in \ker{\rho_\cA}$.

			\textbf{Proof of Claim 1}
			
			Let us define the projections $p_1:K \rightarrow K_1$ and $p_2:K \rightarrow K_2$ as follows:
			$$p_1\left(\begin{matrix}
				u_{1,g}\\
				u_{2,g}\\
				v_{1,g}\\
				v_{2,g}
			\end{matrix}\right)= \left(\begin{matrix}
				u_{1,g}\\
				u_{2,g}\\
			\end{matrix}\right)
			\qquad
			p_2\left(\begin{matrix}
				u_{1,g}\\
				u_{2,g}\\
				v_{1,g}\\
				v_{2,g}
			\end{matrix}\right)= \left(\begin{matrix}
				v_{1,g}\\
				v_{2,g}\\
			\end{matrix}\right).$$
			
			Observe that, in our hypothesis, both the projections $p_1, p_2$ have to be injective. In fact, if one of them, say $p_1$, would be non-injective, then there exists a non-trivial $g \in \ker{p_1}$. This means that the group $H_1=\theta_{\sigma_2}(\ker{\theta_{\sigma_1}})$ defined above is non-trivial. By Proposition $\ref{possibleCasesIso}$, we then would have $\theta_{\sigma}(\ker{\rho_\cA})\isom \bZ^4$, which is a contradiction.
			
			Therefore, $p_1$ and $p_2$ are isomorphisms. Thus, we can define the isomorphism $\varphi:=p_2\circ p_1^{-1}:K_1 \rightarrow K_2$ which maps 
			$$\left(\begin{matrix}
				u_{1,g}\\
				u_{2,g}\\
			\end{matrix}\right)
			\mapsto
			\left(\begin{matrix}
				v_{1,g}\\
				v_{2,g}\\
			\end{matrix}\right).$$
			
			Then, there exists a matrix $M=\left(\begin{matrix}
				\alpha & \beta\\
				\gamma & \delta
			\end{matrix}\right) \in \textrm{GL}_2(\bQ)$, such that
			$$\left(\begin{matrix}
				v_{1,g}\\
				v_{2,g}
			\end{matrix}\right) = M \cdot \left(\begin{matrix}
				u_{1,g}\\
				u_{2,g}
			\end{matrix}\right) = \left(\begin{matrix}
				\alpha u_{1,g} + \beta u_{2,g}\\
				\gamma u_{1,g} + \delta u_{2,g}
			\end{matrix}\right) \qquad \textrm{for all } g \in \ker{\rho_\cA}.$$
			
			To find this matrix $M$, let us choose a basis			
			$$\left(\begin{matrix}
				u_{1,g}\\
				u_{2,g}
			\end{matrix}\right), \quad \left(\begin{matrix}
			u_{1,h}\\
			u_{2,h}
			\end{matrix}\right)$$
			of $K_1$, where $g,h$ are two fixed elements of $\ker{\rho_\cA}$. By imposing the conditions
			$$M\cdot\left(\begin{matrix}
				u_{1,g}\\
				u_{2,g}
			\end{matrix}\right) = \left(\begin{matrix}
				v_{1,g}\\
				v_{2,g}
			\end{matrix}\right), \quad \textrm{and} \quad M\cdot\left(\begin{matrix}
			u_{1,h}\\
			u_{2,h}
			\end{matrix}\right)= \left(\begin{matrix}
				v_{1,h}\\
				v_{2,h}
			\end{matrix}\right),$$
			we obtain a system where the unknowns are the coefficients of $M$. By solving this, we find $M$. Since $\varphi:K_1 \rightarrow K_2$ is an isomorphism, we have that the relation
			$$M\cdot\left(\begin{matrix}
				u_{1,g}\\
				u_{2,g}
			\end{matrix}\right) = \left(\begin{matrix}
				v_{1,g}\\
				v_{2,g}
			\end{matrix}\right)$$
			is true for all $g \in \ker{\rho_\cA}$.

			\textbf{Claim 2: } The matrix $M$ is of the form $$M=\left(\begin{matrix} \alpha & 0\\ 0 & \alpha \end{matrix}\right),$$ where $\alpha \in\bQ$.

			\textbf{Proof of Claim 2}

			Now, let us choose an element $h \in \ker{\rho_\cA}$ and an element $g \in \pi_1(B)$. Let us use the following notation
			$$\theta_{\sigma_1}(h)=\left(\begin{matrix} u_{1,h} \\ u_{2,h}\end{matrix}\right)
			\qquad
			\theta_{\sigma_2}(h)=\left(\begin{matrix} v_{1,h} \\ v_{2,h}\end{matrix}\right),$$
			and let us consider the following periods
			$$\omega_{h,\sigma_1}:= u_{1,h}\omega_1 + u_{2,h}\omega_2,
			\qquad
			\omega_{h,\sigma_2}:= v_{1,h}\omega_1 + v_{2,h}\omega_2.$$
			Moreover, we will indicate with $\omega_{g,\sigma_1}, \omega_{g,\sigma_2}$ the variation of $\log_{\sigma_1}, \log_{\sigma_2}$ along $g$, respectively.
			
			Finally, let us consider the element $h':=ghg^{-1} \in \ker{\rho_\cA}$ and use an analogous notation as above, i.e.
			$$\omega_{h',\sigma_1}:= u_{1,h'}\omega_1 + u_{2,h'}\omega_2,
			\qquad
			\omega_{h',\sigma_2}:= v_{1,h'}\omega_1 + v_{2,h'}\omega_2.$$	
			
			If we look at the action of $h'$ on the determination of $\log_{\sigma_1}$ we obtain
			\begin{align*}
				\log_{\sigma_1} \xrightarrow{g^{-1}} \log_{\sigma_1} - c_{g^{-1}}(\omega_{g,\sigma_1}) &\xrightarrow{h} \log_{\sigma_1} - c_{g^{-1}}(\omega_{g,\sigma_1}) + \omega_{h,\sigma_1}\\ &\xrightarrow{g} \log_{\sigma_1} + c_g(\omega_{h,\sigma_1}).
			\end{align*}
			
			At the same way, if we look at the action of $h'$ on the determination of $\log_{\sigma_2}$ we obtain
			$$\log_{\sigma_2} \xrightarrow{h'} \log_{\sigma_2} + c_g(\omega_{h,\sigma_2}).$$
			We can resume this by the following equations:
			$$\omega_{h',\sigma_1} = c_g(\omega_{h,\sigma_1}), \qquad \omega_{h',\sigma_2} = c_g(\omega_{h,\sigma_2}).$$ In terms of coordinates, this means
			\begin{equation}\label{firstIso}
				\left(\begin{matrix}
					u_{1,h'}\\
					u_{2,h'}
				\end{matrix}\right) = \rho_\cE(g)\cdot \left(\begin{matrix}
					u_{1,h}\\
					u_{2,h}
				\end{matrix}\right),
				\qquad
				\left(\begin{matrix}
					v_{1,h'}\\
					v_{2,h'}
				\end{matrix}\right) = \rho_\cE(g)\cdot \left(\begin{matrix}
					v_{1,h}\\
					v_{2,h}
				\end{matrix}\right).
			\end{equation}
			
			Moreover, since $h, h' \in \ker{\rho_\cA}$, by Claim $1$ we have that
			\begin{equation}\label{secondIso}
				M\cdot \left(\begin{matrix}
					u_{1,h}\\
					u_{2,h}
				\end{matrix}\right) = \left(\begin{matrix}
					v_{1,h}\\
					v_{2,h}
				\end{matrix}\right),
				\qquad
				M\cdot \left(\begin{matrix}
					u_{1,h'}\\
					u_{2,h'}
				\end{matrix}\right) = \left(\begin{matrix}
					v_{1,h'}\\
					v_{2,h'}
				\end{matrix}\right).
			\end{equation}
			
			Now, we are ready to put all together: by $(\ref{firstIso})$ and $(\ref{secondIso})$ we obtain
			\begin{equation}\label{thirdIso}
				\left( M\rho_\cE(g) - \rho_\cE(g) M\right)\cdot \left(\begin{matrix}
					u_{1,h}\\
					u_{2,h}
				\end{matrix}\right)=0,
			\end{equation}
			for all $g \in \pi_1(B)$ and $h \in \ker{\rho_\cA}$.
			Since $\rho_\cE(\pi_1(B))$ is Zariski-dense in $\textrm{SL}_2(\bZ)$, this last relation has to be true for every matrix $A \in \textrm{SL}_2(\bZ)$ in place of $\rho_\cE(g)$ (observe that the relation does not depend on $u_{1,g}, u_{2,g}$).
			Let us choose $$A=\left( \begin{matrix} 1 & 0\\ 2 & 1 \end{matrix} \right).$$ Then for each $h \in \ker{\rho_\cA}$, the relation $(\ref{thirdIso})$ gives us the following equation
			\begin{align*}
				\left(\begin{matrix} 2\beta & 0\\ 2(\delta-\alpha) & -2\beta \end{matrix}\right)
				\cdot
				\left( \begin{matrix} u_{1,h}\\ u_{2,h} \end{matrix} \right)=0.
			\end{align*}
			Since $\theta_{\sigma_1}(\ker{\rho_\cE})\isom \bZ^2$ and $\ker{\rho_\cA}=\ker{\rho_\cE}$, there exists $h \in \ker{\rho_\cA}$ such that $u_{1,h}\neq 0$. It follows that $$\beta=0, \qquad \alpha=\delta.$$
			By choosing $A=\left(\begin{matrix} 1 & 2\\ 0 &1\end{matrix}\right)$ we also obtain $\gamma=0$. So the matrix $\sigma$ have the following form:
			$$M=\left(\begin{matrix} \alpha & 0\\ 0 & \alpha \end{matrix}\right),$$ where $\alpha \in\bQ$, say $\alpha:=\frac{m}{n}$.

			\textbf{End of the proof}
			
			Claim $1$ and Claim $2$ mean that $v_{1,h}=\alpha u_{1,h}, v_{2,h}=\alpha u_{2,h}$ for all $h \in \ker{\rho_\cA}$.
			In other words the logarithm of $\sigma_2$ has the following variation under the action of each $h \in \ker{\rho_\cA}$:
			$$ \log_{\sigma_2} \mapsto \log_{\sigma_2} + \alpha u_{1,h} \omega_1 + \alpha u_{2,h} \omega_2.$$
			Now, let us consider the sections $m\cdot \sigma_1$ and $n\sigma_2$. Observe that for each $h \in \ker{\rho_\cA}$ we have
			\begin{align*}
				\log_{m\sigma_1}=m\log_{\sigma_1} &\mapsto m\log_{\sigma_1} + mu_{1,h} \omega_1 + mu_{2,h} \omega_2,\\
				\log_{n\sigma_2} = n\log_{\sigma_2}&\mapsto n\log_{\sigma_2} + mu_{1,h} \omega_1 + mu_{2,h} \omega_2.
			\end{align*}
			Therefore, we can define the section $\widetilde{\sigma}:= m\sigma_1 - n\sigma_2$ of $\cE\rightarrow B$. Observe that for each $h \in \ker{\rho_\cA}=\ker{\rho_\cE}$ we have
			$$\log_{\widetilde{\sigma}} \mapsto \log_{\widetilde{\sigma}}.$$
			This means that $\theta_{\widetilde{\sigma}}(\ker{\rho_\cE})$ is trivial. By Theorem $\ref{mainTheorem}$, it follows that $\widetilde{\sigma}$ is a torsion section, i.e. we have $$k\widetilde{\sigma}=0,$$
			and this means that $\sigma_1$ and $\sigma_2$ are linearly dependent over $\bZ$; this contradiction concludes the proof.
			
		\end{enumerate}
	\end{proof}
	
	\begin{rem}
		Let us consider the abelian scheme $\cA:=\cE\times_B\cE \rightarrow B$. It is well known that a pair $(P_1,P_2) \in \cA_b$ is contained in a proper group-subscheme of $\cA \rightarrow B$ if and only if there exist $n_1, n_2 \in \bZ$ such that $n_1P_1+n_2P_2=0$. In other words, saying that the image of a section $\sigma=(\sigma_1, \sigma_2)$ is not contained in a proper group-subscheme is equivalent to saying that $\sigma_1, \sigma_2$ are linearly independent over $\bZ$. Thus, Theorem $\ref{squareProduct}$ proves Conjecture $\ref{conj}$ for the case $\cA=\cE\times_B \cE$.
	\end{rem}
	
	By Proposition $\ref{isogenyInvariance}$, we deduce the theorem in the case of product of two isogenous elliptic schemes, which reads as follows:
	
	\begin{thm}\label{isoProduct}
		Let $\sigma_1: B \rightarrow \cE_1, \sigma_2: B \rightarrow \cE_2$ be rational sections of two elliptic schemes such that at least one of them is non-torsion. Suppose that there exists an isogeny $\phi:\cE_1 \rightarrow \cE_2$. Let us consider the abelian scheme $\pi: \cA:=\cE_1\times_B\cE_2 \rightarrow B$ endowed with the (non-torsion) section $\sigma=(\sigma_1,\sigma_2)$. We have the following situation:
		\begin{enumerate}
			\item if $\phi\circ\sigma_1, \sigma_2$ are linearly dependent over $\bZ$, the cover $B_\sigma\rightarrow B^*$ has infinite degree and its Galois group is isomorphic to $\bZ^2$;
			
			\item if $\phi\circ\sigma_1, \sigma_2$ are linearly independent over $\bZ$, the cover $B_\sigma\rightarrow B^*$ has infinite degree and its Galois group is isomorphic to $\bZ^4$.
		\end{enumerate}
	\end{thm}

	\begin{exe}\label{isoExa}
		Let us consider the following two algebraic sections of the Legendre scheme:
		$$\sigma_1(\lambda)=\left( 2, \sqrt{2(2-\lambda)} \right), \qquad \sigma_2(\lambda)=\left( \lambda+1, \sqrt{\lambda(\lambda+1)} \right).$$
		The base $B$ on which the two sections become well-defined may be taken as the (ramified) cover of $\bP_1-\{0,1,\infty\}$ defined by taking the square roots of $2-\lambda$ and of $\lambda(\lambda+1)$. This cover has degree $4$ and is ramified above $\lambda=2$ and $\lambda=-1$. Let us define the elliptic scheme $\cE \rightarrow B$ obtained extending the Legendre scheme by base change to $B$ and consider the abelian family $\cA:=\cE\times_B\cE \rightarrow B$ (observe that the abelian family $\cA$ is obtained as the fiber square of the Legendre scheme, extended by base change to $B$). The above two sections give a rational section $\sigma:B \rightarrow \cA$, whose components we continue to denote by $\sigma_1, \sigma_2$.
		
		Note that none of the sections is identically torsion: in fact, it is known that every torsion section can be defined over a base which is an unramified cover of $\bP_1-\{0,1,\infty\}$, whereas to define $\sigma_1$ (resp. $\sigma_2$) we need the base $B$ to be ramified (at least) above the point $2$ (resp. $-1$) of $\bP_1-\{0,1,\infty\}$. Moreover, the fact that the minimal ramification necessary to define $\sigma_1, \sigma_2$ is different for the two sections, implies that they are linearly independent over $\bZ$. To prove this assertion, let us look at the monodromy action on a possible relation $n_1\sigma_1 + n_2\sigma_2=0$ and observe that for these sections a different choice of the square root would merely change sign to the section. Thus, if we look at the monodromy action induced by a ``small loop'' turning around $2$ in $\bP_1-\{0,1,\infty\}$ on the dependence relation, we change sign to $\sigma_1$ but leave unchanged $\sigma_2$ and so obtain $n_1=0$. Analogously, we obtain $n_2=0$.
		
		Theorem $\ref{isoProduct}$ yields that the relative monodromy group of the logarithm of $\sigma$ with respect to periods is isomorphic to $\bZ^4$. Moreover, we can say something about an explicit loop which leaves unchanged periods but not logarithm. In \cite{Tro}, we have constructed such a loop $\Gamma$ for the logarithm of $\sigma_1$. This loop $\Gamma$ is one of the loops we are looking for: in fact it obviously also works for the logarithm of $\sigma$ with respect to the periods of the abelian scheme, since the periods of $\cA \rightarrow B$ are determined by the periods of $\cE\rightarrow B$.
	\end{exe}

\subsection{Case 2: product of non-isogenous elliptic schemes}

	In this last section, we now formulate a result on the monodromy of the logarithm of a section $\sigma:B \rightarrow \cA=\cE_1 \times_B \cE_2$, in the case in which $\cE_1, \cE_2$ are not isogenous.
	
	In what follows we will make the following identifications: if $g \in \ker{\rho_{\cE_i}}$ we identify $\theta_{\sigma_i}(g) \equiv w_{i,g} \in \bZ^2$; if $ g \in \ker{\rho_\cA}$ we identify $\theta_\sigma(g) \equiv w_g\in \bZ^4$. Moreover, we denote by $\omega_{1,\cE_i}, \omega_{2,\cE_i}$ the periods of $\cE_i \rightarrow B$.
	 
	Recall that unlike the case in which $\cE_1, \cE_2$ are isogenous schemes, in this case the representations $\rho_{\cE_1}$ and $\rho_{\cE_2}$ are not conjugate. Rather, the image $\rho_\cA(\pi_1(B))$ is Zariski-dense in $\textrm{SL}_2(\bZ)\times \textrm{SL}_2(\bZ)$ (see Theorem $\ref{isogenyThm}$).
	 
	Now, we are ready for the results of this section.
	 
	\begin{lemma}\label{NoCyclicGroup}
		Let $H \unlhd G:=\pi_1(B)$ be a normal subgroup of $\pi_1(B)$. If $H \subset \ker{\rho_{\cE_1}}$ (resp. $H \subset \ker{\rho_{\cE_2}}$), then $\theta_{\sigma_1}(H)$ (resp. $\theta_{\sigma_2}(H)$) is isomorphic to either $\{0\}$ or $\bZ^2$.
	\end{lemma}
	
		\begin{proof}
	We will give the proof only for the case $H \subset \ker{\rho_{\cE_1}}$, since the other case is analogous. Since we are going to work only with the scheme $\cE_1 \rightarrow B$, let us denote the periods $\omega_{1,\cE_1}, \omega_{2,\cE_1}$ simply by $\omega_1, \omega_2$ in this lemma.
	
		At first, since $H \subset \ker{\rho_{\cE_1}}$ observe that $\theta_{\sigma_1}(H)$ is a subgroup of $\bZ^2$; so it is either isomorphic to $\{0\}$ or $\bZ$ or $\bZ^2$. We want to prove that the case $\bZ$ is excluded. Suppose by contradiction that $\theta_{\sigma_1}(H)$ is infinite cyclic: this means that for every $h \in H$, the logarithm $\log_{\sigma_1}$ of $\sigma_1$ is transformed by $h$ as $$\log_{\sigma_1} \xmapsto{h} \log_{\sigma_1} + \chi(h)\omega_{\sigma_1},$$ for a fixed non-zero period $\omega_{\sigma_1}$ and a homomorphism $\chi:H \rightarrow \bZ$. In particular, let us choose $h$ such that $\chi(h)=1$. Recall that, for $g \in G=\pi_1(B)$ the logarithm $\log_{\sigma_1}$ will be sent by $g$ to a new determination of the form $$\log_{\sigma_1} + u_{1,g}\omega_1 + u_{2,g} \omega_2,$$ where $u_{1,g}, u_{2,g}$ are integers. Recall that the group $G_1=\rho_{\cE_1}(G)$ acts irreducibly on the lattice of periods, since it is Zariski-dense in $\textrm{SL}_2(\bZ)$. Then there exists $g \in G$ such that $\omega_{\sigma_1}$ is not an eigenvector of $\rho_{\cE_1}(g)$. Since $H \unlhd G$, we have $h'=g^{-1}hg \in H$, where $g, h$ are the ones just considered. Let us calculate the action of the element $h'=g^{-1}hg$. We have
		\begin{align*}
			\log_{\sigma_1} &\xmapsto{g} \log_{\sigma_1} + u_{1,g} \omega_1 + u_{2,g} \omega_2 \xmapsto{h} \log_{\sigma_1} + u_{1,g}\omega_1 + u_{2,g} \omega_2 + \omega_{\sigma_1}\\
			& \xmapsto{g^{-1}} \log_{\sigma_1} +\rho_{\cE_2}(g^{-1})\omega_{\sigma_1}.
		\end{align*}
		In other words this means $\rho_{\cE_1}(g^{-1})\omega_{\sigma_1}=\chi(h')\omega_{\sigma_1}$, but this is a contradiction since $\omega_{\sigma_1}$ is not an eigenvector of $\rho_{\cE_1}(g)$ (nor of $\rho_{\cE_1}(g^{-1})$). This concludes the proof.
	\end{proof}

Now, suppose that both $\sigma_1$ and $\sigma_2$ are non-torsion and let's take a look at the difference with the case where the two schemes are isogenous. In both cases, we can use Theorem $\ref{mainTheorem}$ for $\sigma_1, \sigma_2$ and obtain that
$$\theta_{\sigma_1}(\ker{\rho_{\cE_1}})\isom \theta_{\sigma_2}(\ker{\rho_{\cE_2}}) \isom \bZ^2.$$ 
The problem is that when the schemes are isogenous, we have 
$$\ker{\rho_{\cE_1}}=\ker{\rho_{\cE_2}}=\ker{\rho_\cA}$$ (since the monodromy groups of periods are conjugate). Thus we deduce immediately that $\theta_\sigma(\ker{\rho_\cA}) \neq 0$.

Instead, when the two schemes are not isogenous, the group $\ker{\rho_\cA}=\ker{\rho_{\cE_1}} \cap \ker{\rho_{\cE_2}}$ can be smaller than $\ker{\rho_{\cE_1}}$ and $\ker{\rho_{\cE_2}}$. Therefore, this time Theorem $\ref{mainTheorem}$ does not allow us to conclude directly that $\theta_\sigma(\ker{\rho_\cA})\neq 0$. However, the conclusion is still true and we prove it (and a little more) in the next proposition.

\begin{prop}\label{nonTrivial}
	If $\sigma_1, \sigma_2$ are both non-torsion, then $\theta_\sigma(\ker{\rho_\cA})\neq 0$. Moreover, we have
	$$\theta_{\sigma_1}(\ker{\rho_\cA})\isom \bZ^2, \qquad \theta_{\sigma_2}(\ker{\rho_\cA})\isom \bZ^2.$$
\end{prop}

\begin{proof}
	By Theorem $\ref{mainTheorem}$ in the case of $\cE_1 \rightarrow B$ and $\cE_2 \rightarrow B$, we have that there exist two elements $g_1, g_2 \in \pi_1(B)$ such that
	
$$\theta_{\sigma_1}(g_1)=\left(
\begin{matrix}
I_2 & \begin{matrix} u_{1,g_1} \\ u_{2,g_1} \end{matrix}\\
0 & 1
\end{matrix}
\right),
\qquad
\theta_{\sigma_2}(g_2)=\left(
\begin{matrix}
I_2 & \begin{matrix} v_{1,g_2} \\ v_{2,g_2} \end{matrix}\\
0 & 1
\end{matrix}
\right),$$

where $\left(\begin{matrix} u_{1,g_1} \\ u_{2,g_1} \end{matrix}\right)$ and $\left(\begin{matrix} v_{1,g_2} \\ v_{2,g_2} \end{matrix}\right)$ are non-zero vectors. By looking at the representation $\theta_{\sigma}$ we obtain the following two matrices:

$$\theta_\sigma(g_1)=\left(
\begin{matrix}
\begin{matrix} I_2 & 0\\
0 & \rho_2(g_1)\end{matrix} & \begin{matrix} u_{1,g_1} \\ u_{2,g_1} \\ v_{1,g_1} \\ v_{2,g_1} \end{matrix}\\
0 & 1
\end{matrix}
\right),
\qquad
\theta_\sigma(g_2)=\left(
\begin{matrix}
\begin{matrix} \rho_1(g_2) & 0\\
0 & I_2 \end{matrix} & \begin{matrix} u_{1,g_2} \\ u_{2,g_2} \\ v_{1,g_2} \\ v_{2,g_2} \end{matrix}\\
0 & 1
\end{matrix}
\right).$$

If we compute $\theta_\sigma(g_1g_2g_1^{-1}g_2^{-1})$ we obtain
$$\theta_\sigma(g_1g_2g_1^{-1}g_2^{-1})=\left(
\begin{matrix}
I_2 & 0 & \left(\begin{matrix} u_{1,g_1} \\ u_{2,g_1}\end{matrix}\right) - \rho_{\cE_1}(g_2) \left(\begin{matrix} u_{1,g_1} \\ u_{2,g_1}\end{matrix}\right) \\ 
0 & I_2 & \rho_{\cE_2}(g_1) \left(\begin{matrix} v_{1,g_2} \\ v_{2,g_2}\end{matrix}\right) - \left(\begin{matrix} v_{1,g_2} \\ v_{2,g_2}\end{matrix}\right)\\
0 & 0 & 1
\end{matrix}
\right).$$
In particular, this proves that $g_1g_2g_1^{-1}g_2^{-1} \in \ker{\rho_\cA}$ for each $g_1 \in \ker{\rho_{\cE_1}}, g_2 \in \ker{\rho_{\cE_2}}$.

Since $\theta_{\sigma_1}(\ker{\rho_{\cE_1}}) \isom \bZ^2$ and $\theta_{\sigma_2}(\ker{\rho_{\cE_2}}) \isom \bZ^2$, we can choose $g_1\in \ker{\rho_{\cE_1}}, g_2 \in \ker{\rho_{\cE_2}}$ in such a way that
$$\left(\begin{matrix} u_{1,g_1} \\ u_{2,g_1}\end{matrix}\right) - \rho_{\cE_1}(g_2) \left(\begin{matrix} u_{1,g_1} \\ u_{2,g_1}\end{matrix}\right) \neq 0.$$
Then, we can also choose other $g_1, g_2$ such that
$$\rho_{\cE_2}(g_1) \left(\begin{matrix} v_{1,g_2} \\ v_{2,g_2}\end{matrix}\right) - \left(\begin{matrix} v_{1,g_2} \\ v_{2,g_2}\end{matrix}\right) \neq 0.$$

Since $g_1g_2g_1^{-1}g_2^{-1} \in \ker{\rho_\cA}$ for each $g_1 \in \ker{\rho_{\cE_1}}, g_2 \in \ker{\rho_{\cE_2}}$, we conclude that
$$\theta_\sigma(\ker{\rho_\cA}), \theta_{\sigma_1}(\ker{\rho_\cA}), \theta_{\sigma_2}(\ker{\rho_\cA}) \neq 0.$$
In particular, by Lemma $\ref{NoCyclicGroup}$ we have
$$\theta_{\sigma_1}(\ker{\rho_\cA}) \isom \bZ^2, \qquad \theta_{\sigma_2}(\ker{\rho_\cA})\isom \bZ^2.$$
\end{proof}

Let us define
	 $$H_1 := \theta_{\sigma_2}(\ker{\theta_{\sigma_1}} \cap \ker{\rho_{\cA}}), \quad H_2 := \theta_{\sigma_1}(\ker{\theta_{\sigma_2}} \cap \ker{\rho_{\cA}}).$$
	 Since $\ker{\theta_{\sigma_1}} \cap \ker{\rho_{\cA}}$ is a normal subgroup of $G$ which is contained in $\ker{\rho_{\cE_2}}$, by Lemma $\ref{NoCyclicGroup}$ we have that $H_1$ is isomorphic either to $\{0\}$ or to $\bZ^2$. The same is true for $H_2$.
	
	\begin{prop}\label{possibleCases}
		If both $\sigma_1, \sigma_2$ are non-torsion, then the group $\theta_\sigma(\ker{\rho_\cA})$ is isomorphic to either $\bZ^2$ or $\bZ^4$.
	\end{prop}

	\begin{proof}
	
	\textbf{Claim: } $\theta_\sigma(\ker{\rho_\cA}) \neq \bZ$.\\
	
	\textbf{Proof of the Claim: } 
	
	Let us suppose $\textrm{rank }{\theta_\sigma(\ker{\rho_\cA})} <2$. Observe that 
	\begin{align*}
		\textrm{rank }{\theta_{\sigma_1}(\ker{\rho_\cA})} \le \textrm{rank }{\theta_\sigma(\ker{\rho_\cA})},\\
		\textrm{rank }{\theta_{\sigma_2}(\ker{\rho_\cA})} \le \textrm{rank }{\theta_\sigma(\ker{\rho_\cA})}.
	\end{align*}
	
	Thus, by Lemma $\ref{NoCyclicGroup}$, this implies that
	$$\theta_{\sigma_1}(\ker{\rho_\cA}) \isom \{0\}, \qquad \theta_{\sigma_2}(\ker{\rho_\cA})\isom\{0\},$$
	which is in contradiction with Proposition $\ref{nonTrivial}$.
	
	\textbf{Final proof}
		
		Since both the sections are non-torsion, by Proposition $\ref{nonTrivial}$ and by the Claim, we have that
		$$2 \le \textrm{rank}\, \theta_\sigma(\ker{\rho_\cA}) \le 4.$$ 
		By Lemma $\ref{NoCyclicGroup}$, we only have two possibilities for $H_1$, i.e. $H_1 \isom \{0\}$ or $H_1\isom \bZ^2$.\\
		
		\textbf{Case 1: $H_1\isom\{0\}$}\\
		
		The condition $H_1\isom\{0\}$ means that for each $g \in \ker{\rho_\cA}$ if $u_{1,g}=u_{2,g}=0$, then $v_{1,g}=v_{2,g}=0$, where the notation is the same as above. Let us prove that $\textrm{rank}\, \theta_\sigma(\ker{\rho_\cA})=2$ by proving that any three elements of the form $w_g, w_h, w_k \in \theta_\sigma(\ker{\rho_\cA})$ are linearly dependent on $\bZ$.
		
		Since $\theta_{\sigma_1}(\ker{\rho_{\cE_1}}) \isom \bZ^2$, given any three elements $g,h,k \in \ker{\rho_\cA} \subset \ker{\rho_{\cE_1}}$, there always exists $n_g, n_h, n_k \in \bZ$, not all zero, such that
		$$n_g\left(\begin{matrix}
			u_{1,g}\\
			u_{2,g}
		\end{matrix}\right) + n_h\left(\begin{matrix}
			u_{1,h}\\
			u_{2,h}
		\end{matrix}\right) + n_k\left(\begin{matrix}
			u_{1,k}\\
			u_{2,k}
		\end{matrix}\right)= \left(\begin{matrix}
			0\\
			0
		\end{matrix}\right).$$
		Thus we have
		$$\theta_\sigma(k^{n_k}h^{n_h}g^{n_g})=\left(\begin{matrix}
			\begin{matrix}
				\textrm{I}_4
			\end{matrix} & \begin{matrix}
			0\\
			0\\
			n_gv_{1,g} + n_hv_{1,h} +n_kv_{1,k}\\
			n_gv_{2,g} + n_hv_{2,h} +n_kv_{2,k}
		\end{matrix}\\
		0 & 1
		\end{matrix}\right).$$
		The condition $H_1\isom\{0\}$ implies that
		$$n_g\left(\begin{matrix}
			v_{1,g}\\
			v_{2,g}
		\end{matrix}\right) + n_h\left(\begin{matrix}
			v_{1,h}\\
			v_{2,h}
		\end{matrix}\right) + n_k\left(\begin{matrix}
			v_{1,k}\\
			v_{2,k}
		\end{matrix}\right) = \left(\begin{matrix}
			0\\
			0
		\end{matrix}\right).$$
		Therefore any three elements of $\bZ^4$ of the form $w_g, w_h, w_k$ are linearly dependent on $\bZ$, so $\theta_\sigma(\ker{\rho_\cA})\isom \bZ^2$.\\
		
		\textbf{Case 2: $H_1\isom \bZ^2$}\\
		Since $H_1\isom \bZ^2$ we can consider a $\bZ$-basis for it and the following corresponding elements of $\theta_\sigma(\ker{\rho_\cA})$:
		$$w_3=\left(\begin{matrix}
			0\\
			0\\
			v_{1,k}\\
			v_{2,k}
		\end{matrix}\right), w_4=\left(\begin{matrix}
			0\\
			0\\
			v_{1,l}\\
			v_{2,l}
		\end{matrix}\right)
		\qquad \textrm{where } k,l \in \ker{\theta_{\sigma_1}}\cap \ker{\rho_\cA}\subset\ker{\rho_\cA}.$$
		Since both sections are non-torsion, by Proposition $\ref{nonTrivial}$ we have that $\theta_{\sigma_1}(\ker{\rho_{\cA}})\isom \bZ^2$. Therefore, let us choose a $\bZ$-basis $\left(\begin{matrix}
			u_{1,g}\\
			u_{2,g}
		\end{matrix}\right), \left(\begin{matrix}
			u_{1,h}\\
			u_{2,h}
		\end{matrix}\right)$ of it, where $g,h \in \ker{\rho_\cA}$, and consider the corresponding elements of $\theta_\sigma(\ker{\rho_\cA})$:
		$$z_1=\left(\begin{matrix}
			u_{1,g}\\
			u_{2,g}\\
			v_{1,g}\\
			v_{2,g}
		\end{matrix}\right), z_2=\left(\begin{matrix}
			u_{1,h}\\
			u_{2,h}\\
			v_{1,h}\\
			v_{2,h}
		\end{matrix}\right).$$

		With an appropriate linear combination of $z_1, z_2, w_3, w_4$ we obtain that
		$$w_1=\left(\begin{matrix}
			u_{1,g}\\
			u_{2,g}\\
			0\\
			0
		\end{matrix}\right), w_2=\left(\begin{matrix}
			u_{1,h}\\
			u_{2,h}\\
			0\\
			0
		\end{matrix}\right)$$
		are elements of $\theta_\sigma(\ker{\rho_\cA})$. Moreover, $w_1, w_2, w_3, w_4$ are linearly independent over $\bZ$. Thus $\theta_\sigma(\ker{\rho_\cA}) \isom \bZ^4$.\\		
\end{proof}

\subsubsection{Main Theorem}
		
		Now, recall we are considering an abelian scheme $\cA:= \cE_1 \times_B \cE_2$, where $\cE_1$ and $\cE_2$ are not isogenous.
	
	\begin{thm}\label{nonIsoProduct}
		Let $\sigma_i: B \rightarrow \cE_i, i=1,2$ be rational sections of two non-isogenous elliptic schemes and suppose they are not both torsion sections. Let us consider the abelian scheme $\pi: \cA:=\cE_1\times_B\cE_2 \rightarrow B$ endowed with the (non-torsion) section $\sigma=(\sigma_1,\sigma_2)$. We have the following situation:
		\begin{enumerate}
			\item if one between $\sigma_1, \sigma_2$ is a torsion section, the cover $B_\sigma\rightarrow B^*$ has infinite degree and its Galois group is isomorphic to $\{0\}$ or to $\bZ^2$;
			
			\item if neither $\sigma_1$ nor $\sigma_2$ is a torsion section, the cover $B_\sigma\rightarrow B^*$ has infinite degree and its Galois group is isomorphic to $\bZ^4$.
		\end{enumerate}
	\end{thm}

	\begin{proof}
		\begin{enumerate}
			\item Suppose that one between $\sigma_1, \sigma_2$ is a torsion section, say for example $\sigma_1$. Suppose that $\theta_\sigma(\ker{\rho_\cA})\neq 0$ and prove that $\theta_\sigma(\ker{\rho_\cA}) \isom \bZ^2$. Since $\sigma_1$ is torsion, we have $\theta_{\sigma_1}(\ker{\rho_{\cE_1}})=0$ (see Remark $\ref{torsionSection}$); in particular, this implies that $\theta_{\sigma_1}(\ker{\rho_\cA})=0$. The two conditions
			$$\theta_\sigma(\ker{\rho_\cA})\neq 0, \qquad \textrm{ and } \qquad \theta_{\sigma_1}(\ker{\rho_\cA})=0,$$
			imply that
			$$\theta_\sigma(\ker{\rho_\cA})\isom \theta_{\sigma_2}(\ker{\rho_\cA}), \qquad \textrm{ and } \qquad \theta_{\sigma_2}(\ker{\rho_\cA})\neq 0.$$
			By Lemma $\ref{NoCyclicGroup}$, we have that
			$$\theta_{\sigma_2}(\ker{\rho_\cA})\isom \bZ^2.$$
			Thus, it follows that $\theta_\sigma(\ker{\rho_\cA})\isom \bZ^2$ and the first part is proved.

			\item Now, let $\sigma_1, \sigma_2$ be both non-torsion sections and let us suppose by contradiction that $\theta_\sigma(\ker{\rho_\cA})$ is not isomorphic to $\bZ^4$. By Proposition $\ref{possibleCases}$, we can only have the case
			$$\theta_\sigma(\ker{\rho_\cA}) \isom \bZ^2.$$
			With the same calculations explicited in the case of product of isogenous curves, we obtain that there exists a matrix $M \in \textrm{GL}_2(\bQ)$ such that
			\begin{equation}\label{star}
				M\cdot \left(\begin{matrix}
					u_{1,g}\\
					u_{2,g}\end{matrix}\right) = \left(\begin{matrix}
					v_{1,g}\\
					v_{2,g}\end{matrix}\right)
			\end{equation}
			for all $g \in \ker{\rho_\cA}$.

		\textbf{Claim: } The matrix $M$ is the zero matrix, i.e. $M=0$.

			\textbf{Proof of Claim}

			Now, let us choose an element $h \in \ker{\rho_\cA}$ and an element $g \in \pi_1(B)$. Let us use the following notation
			$$\theta_{\sigma_1}(h)=\left(\begin{matrix} u_{1,h} \\ u_{2,h}\end{matrix}\right)
			\qquad
			\theta_{\sigma_2}(h)=\left(\begin{matrix} v_{1,h} \\ v_{2,h}\end{matrix}\right),$$
			and let us consider the following periods
			$$\omega_{h,\sigma_1}:= u_{1,h}\omega_{1,\cE_1} + u_{2,h}\omega_{2,\cE_1},
			\qquad
			\omega_{h,\sigma_2}:= v_{1,h}\omega_{1,\cE_2} + v_{2,h}\omega_{2,\cE_2}.$$
			Moreover, we will indicate with $\omega_{g,\sigma_1}, \omega_{g,\sigma_2}$ the variation of $\log_{\sigma_1}, \log_{\sigma_2}$ along $g$, respectively.
			
			Finally, let us consider the element $h':=ghg^{-1} \in \ker{\rho_\cA}$ and use an analogous notation as above, i.e.
			$$\omega_{h',\sigma_1}:= u_{1,h'}\omega_{1,\cE_1} + u_{2,h'}\omega_{2,\cE_1},
			\qquad
			\omega_{h',\sigma_2}:= v_{1,h'}\omega_{1,\cE_2} + v_{2,h'}\omega_{2,\cE_2}.$$	
			
			If we look at the action of $h'$ on the determination of $\log_{\sigma_1}$ we obtain
			\begin{align*}
				\log_{\sigma_1} \xrightarrow{g^{-1}} \log_{\sigma_1} - \rho_{\cE_1}(g^{-1})(\omega_{g,\sigma_1}) &\xrightarrow{h} \log_{\sigma_1} - \rho_{\cE_1}(g^{-1})(\omega_{g,\sigma_1}) + \omega_{h,\sigma_1}\\ &\xrightarrow{g} \log_{\sigma_1} + \rho_{\cE_1}(g)(\omega_{h,\sigma_1}).
			\end{align*}
			
			At the same way, if we look at the action of $h'$ on the determination of $\log_{\sigma_2}$ we obtain
			$$\log_{\sigma_2} \xrightarrow{h'} \log_{\sigma_2} + \rho_{\cE_2}(g)(\omega_{h,\sigma_2}).$$
			We can resume this by the following equations:
			$$\omega_{h',\sigma_1} = \rho_{\cE_1}(g)(\omega_{h,\sigma_1}), \qquad \omega_{h',\sigma_2} = \rho_{\cE_2}(g)(\omega_{h,\sigma_2}).$$ In terms of coordinates, this means
			\begin{equation}\label{first}
				\left(\begin{matrix}
					u_{1,h'}\\
					u_{2,h'}
				\end{matrix}\right) = \rho_{\cE_1}(g)\cdot \left(\begin{matrix}
					u_{1,h}\\
					u_{2,h}
				\end{matrix}\right),
				\qquad
				\left(\begin{matrix}
					v_{1,h'}\\
					v_{2,h'}
				\end{matrix}\right) = \rho_{\cE_2}(g)\cdot \left(\begin{matrix}
					v_{1,h}\\
					v_{2,h}
				\end{matrix}\right).
			\end{equation}
			
			Moreover, since $h, h' \in \ker{\rho_\cA}$, by $(\ref{star})$ we have that
			\begin{equation}\label{second}
				M\cdot \left(\begin{matrix}
					u_{1,h}\\
					u_{2,h}
				\end{matrix}\right) = \left(\begin{matrix}
					v_{1,h}\\
					v_{2,h}
				\end{matrix}\right),
				\qquad
				M\cdot \left(\begin{matrix}
					u_{1,h'}\\
					u_{2,h'}
				\end{matrix}\right) = \left(\begin{matrix}
					v_{1,h'}\\
					v_{2,h'}
				\end{matrix}\right).
			\end{equation}
			
			Now, we are ready to put all together: by $(\ref{first})$ and $(\ref{second})$ we obtain
			\begin{equation}\label{third}
				\left( M\rho_{\cE_1}(g) - \rho_{\cE_2}(g) M\right)\cdot \left(\begin{matrix}
					u_{1,h}\\
					u_{2,h}
				\end{matrix}\right)=0,
			\end{equation}
			for all $g \in \pi_1(B)$ and $h \in \ker{\rho_\cA}$.
			
			Since $\rho_\cA(\pi_1(B))$ is Zariski-dense in $\textrm{SL}_2(\bZ)\times \textrm{SL}_2(\bZ)$ (see Theorem $\ref{isogenyThm}$), this last relation has to be true for every pair of matrices $(A,B) \in \textrm{SL}_2(\bZ)\times \textrm{SL}_2(\bZ)$ in place of $(\rho_{\cE_1}(g),\rho_{\cE_2}(g))$ (observe that the relation does not depend on $u_{1,g}, u_{2,g}$).
			Let us choose
			$$A=\left( \begin{matrix} 1 & 2\\ 0 & 1 \end{matrix} \right), 
			\qquad 
			B=\left( \begin{matrix} 1 & 0\\ 0 & 1 \end{matrix} \right).$$
			Then for each $h \in \ker{\rho_\cA}$, the relation $(\ref{third})$ gives us the following equation
			\begin{align*}
				\left(\begin{matrix} 0 & 2\alpha\\ 0 & 2\gamma \end{matrix}\right)
				\cdot
				\left( \begin{matrix} u_{1,h}\\ u_{2,h} \end{matrix} \right)= 0.
			\end{align*}
			
			By Proposition $\ref{nonTrivial}$, since $\sigma_1$ and $\sigma_2$ are both non-torsion, we have $\theta_{\sigma_1}(\ker{\rho_\cA})\isom \bZ^2$. Therefore, there exists $h \in \ker{\rho_\cA}$ such that $u_{2,h}\neq 0$. It follows that $$\alpha=0, \qquad \gamma=0.$$
			
			Now, since $\theta_{\sigma_1}(\ker{\rho_\cA})\isom \bZ^2$ there also exists $h \in \ker{\rho_\cA}$ such that $u_{1,h}\neq 0$. By choosing 
			$$A=\left( \begin{matrix} 1 & 0\\ 2 & 1 \end{matrix} \right), 
			\qquad 
			B=\left( \begin{matrix} 1 & 0\\ 0 & 1 \end{matrix} \right).$$ 
			we also obtain $$\beta=0, \qquad \delta=0.$$ So the matrix $M$ is the zero matrix, i.e. $M=0$.			
			
			\textbf{End of the proof}
			
			This means that $v_{1,h}=v_{2,h}=0$ for all $h \in \ker{\rho_\cA}$, i.e.
			$$\theta_{\sigma_2}(\ker{\rho_\cA})=0.$$
			Since both sections are non-torsion, this is in contradiction with Proposition $\ref{nonTrivial}$. This concludes the proof.
		\end{enumerate}
	\end{proof}

	\begin{rem}
		Observe that if $E_1, E_2$ are complex elliptic curves which are not isogenous, then every proper connected algebraic subgroup of $E_1\times E_2$ is of one of the shapes $0\times E_2$ or $E_1 \times 0$. In other words, if $\cA=\cE_1\times_B\cE_2 \rightarrow B$ is a product of two non-isogenous elliptic schemes, then saying that the image of a section $\sigma=(\sigma_1, \sigma_2)$ is not contained in a proper group-subscheme is equivalent to saying that neither $\sigma_1$ nor $\sigma_2$ is a torsion section. Thus, Theorem $\ref{nonIsoProduct}$ proves Conjecture $\ref{conj}$ in the case of a fibered product of non-isogenous elliptic schemes.
	\end{rem}

	\begin{exe}
		Let us consider the line $B$ in $S^2$  (where $S=\bP_1-\{0,1,\infty\}$) defined by $x+y=2$, which is isomorphic under the first projection to $\bP_1-\{0,1,2,\infty\}$. Let us consider the scheme over $B$ whose fiber over the point $(\lambda,\mu) \in B$ is the product $\cL_\lambda \times \cL_\mu$ of the corresponding Legendre curves; denote it by $\cA \rightarrow B$. This is a product of two non-isogenous elliptic schemes, since the curve $\cL_\lambda$ is not isogenous generically to $\cL_{2-\lambda}$: in fact, if two elliptic schemes are isogenous then their $j$-invariants must have the same poles in $B$; but in this case, the schemes corresponding to $\cL_\lambda, \cL_\mu$ have a different set of bad reduction. We may consider the section $\sigma:B \rightarrow \cA$ given by
		$$\bP_1-\{0,1,2,\infty\}\ni\lambda \mapsto \left( (2,\sqrt{2-\lambda}), (2,\sqrt{2\lambda}) \right),$$
		whose components are non-torsion sections. Theorem $\ref{nonIsoProduct}$ yields that the relative monodromy group of the logarithm of $\sigma$ with respect to periods is isomorphic to $\bZ^4$.
		
		Can we say something about an explicit loop which leaves unchanged periods but not logarithm? We have such a loop $\Gamma_1$ (resp. $\Gamma_2$) for the logarithm of $\sigma_1$ (resp. $\sigma_2$). Unlike the Example $\ref{isoExa}$, this time the loops $\Gamma_1, \Gamma_2$ don't work for the logarithm of $\sigma$, since the periods of the two factors of $\cA \rightarrow B$ are not the same. Anyway, we can obtain such a loop as explained in the proof of Proposition $\ref{nonTrivial}$, i.e. taking the commutator of suitable loops which work for $\log_{\sigma_1}$ and $\log_{\sigma_2}$ (these last can be found looking at the construction in \cite{Tro}).
\end{exe}

\Addresses

 \end{document}